\documentclass[a4paper,12pt]{amsart}
\usepackage{amssymb,amsmath}
\let\amslrcorner\lrcorner
\usepackage{mathabx}
\usepackage{stmaryrd}
\usepackage{amscd,amsthm,amssymb}
\usepackage{enumerate}
\usepackage{color}
\usepackage[all,cmtip]{xy}
\usepackage{hyperref}
\usepackage{mathrsfs} 

\let\lrcorner\amslrcorner

\usepackage{latexsym}

\usepackage{hyperref}
\hypersetup{
    colorlinks=true, 
    linktoc=all,     
   linkcolor=blue,  
}

\def\sideremark#1{\ifvmode\leavevmode\fi\vadjust{\vbox to0pt{\vss
\hbox to 0pt{\hskip\hsize\hskip1em%
\vbox{\hsize2cm\tiny\raggedright\pretolerance10000%
\noindent {\color{red}{#1}}\hfill}\hss}\vbox to8pt{\vfil}\vss}}}%

\def \bbR{\mathbb{R}}
\def\.{\cdot}

\def\d{{\mathrm d}}

\def\la{\langle}
\def\ra{\rangle}
\def\O{\Omega}

\def\t{\tilde}
\def\beq{\begin{equation}}
\def\eeq{\end{equation}}
\def\bea{\begin{eqnarray*}}
\def\eea{\end{eqnarray*}}
\def\beaa{\begin{eqnarray}}
\def\eeaa{\end{eqnarray}}
\def\ba{\begin{array}}
\def\ea{\end{array}}

\def\o{\omega}

\def \L{\mathscr{L}}
\def \bbE{\mathbb{E}}
\def \RM{\mathbb{R}}

\def \CM{\mathbb{C}}

\def\Ric{\mathrm{Ric}}
\def\id{\mathrm{id}}
\def\be{\begin{equation}}
\def\ee{\end{equation}}
\def\tr{\mathrm{tr}}

\def\SL2{\mathfrak{sl}_2(\bbC)} 
\def\so{\mathfrak{so}}
\def\su{\mathfrak{su}}

\def\sp{\mathfrak{sp}}
\def\gg{\mathfrak{g}}

\def\u{\mathfrak{u}}

\def\aut{\mathfrak{aut}}
\def\SU{\mathrm{SU}}

\def\U{\mathrm{U}}

\def\D{\mathbf{b}} 
\def \scrE{\mathscr{E}} 

\def\G{\mathrm{G}}

\def\SO{\mathrm{SO}}
\def\Sp{\mathrm{Sp}}

\def\End{\mathrm{End}}

\def\vol{\mathrm{vol}}
\def\Ker{\mathrm{Ker}}
\def\ad{\mathrm{ad}}

\def\scal{\mathrm{scal}}

\def\grad{\mathrm{grad}}

\def \t5{\frac{1}{\sqrt{5}}}
\def \dh{\dot{h}}
\def\Pr{C}
\def\FN{\mathrm{FN}} 

\DeclareMathOperator{\TT}{TT} 

\DeclareMathOperator{\ric}{\mathrm{ric}}

\def \bfv{\mathbf{v}}
\def \bfq{\mathbf{q}}
\def \bfe{\mathbf{e}}
\def \bbE{\mathbb{E}}

\DeclareMathOperator{\di}{d} 

\DeclareMathOperator{\bbC}{\mathbb{C}}

\newtheorem{pro}{Proposition}[section]
\newtheorem{teo}[pro]{Theorem}
\newtheorem{lema}[pro]{Lemma}
\newtheorem{coro}[pro]{Corollary}
\theoremstyle{definition}
\newtheorem{defi}[pro]{Definition}
\newtheorem{rema}[pro]{Remark}

\newtheorem{conj}[pro]{Conjecture}

\title{Second order Einstein deformations }
\author{Paul-Andi Nagy}
\address[Paul-Andi Nagy]{Center for Complex Geometry \\
Institute for Basic Science(IBS)\\
55 Expo-ro, Yuseong-gu \\
34126 Daejeon, South Korea
}
\email{paulandin@ibs.re.kr}
\author{Uwe Semmelmann}
\address[Uwe Semmelmann]{Institut f\"ur Geometrie und Topologie, Fachbereich Mathematik, Universit\"at Stuttgart, Pfaffenwaldring 57, 70569 Stuttgart, Germany}
\email{Uwe.Semmelmann@mathematik.uni-stuttgart.de}
\date{\today}

\begin{document}
 
\begin{abstract}
We study the integrability to second order of infinitesimal Einstein deformations on compact Riemannian and
in particular on K\"ahler manifolds. We find a new way of expressing the necessary and
sufficient condition for integrability to second order, which also gives a very clear and compact way
of writing the Koiso obstruction. As an application we consider the K\"ahler case, where the condition
can be further simplified and in complex dimension $3$ turns out to be purely algebraic. One of our
main results is the complete and explicit description of integrable to second order infinitesimal Einstein deformations 
on the complex $2$-plane Grassmannian, which also has a quaternion K\"ahler structure.  As a striking consequence 
we find that the symmetric Einstein metric  on the Grassmannian  $ \mathrm{Gr}_2(\bbC^{n+2})$ for $n$  odd  is rigid.
\medskip



\end{abstract}
\maketitle
\tableofcontents
\section{Introduction} \label{intro}

Let $g$ be an Einstein metric on a compact manifold $M$ with Einstein constant $E$, i.e. $\ric^g = E g$, where $\ric^g$ is the Ricci form;  see Section \ref{sect2} 
for definitions and further details. Koiso in \cite{Ko1} (see also \cite[Chapter 12]{Besse}) studied the question whether $g$ admits an
Einstein deformation, i.e. a curve of Einstein metrics through $g$.  The main tool to investigate this question is the
functional $\mathrm{E}(g) : = \ric_g - (\frac1n \int_M \scal_g \vol_g) g$, where the volume  of $g$ is normalised to one. If $g_t$ is a curve of
Einstein metrics of volume one then $\mathrm{E}(g_t)$ and all its derivatives in $t=0$ vanishes. This gives a sequence of obstructions for an arbitrary curve
$g_t$ of metrics with $g=g_0$ to be a curve of Einstein metrics.

Since Einstein metrics are critical points of the total scalar curvature functional the first derivative of the functional $E$ can be written as
$$
\mathrm{E}'_g(h) \;= \;  \left. \frac{\di}{\di\!t}\right|_{t=0} \mathrm{E}(g_t) \;  = \;  \left. \frac{\di}{\di\!t}\right|_{t=0} (\ric^{g_t} - Eg_t)
$$
where  $g_t$ is  a curve of metrics such that  $g_0 = g$ is  Einstein with Einstein constant $E$ and $h = g^{-1}\dot g$. Here the dot denotes the
derivative at $t=0$.
We call $g_t$  a first order Einstein deformation if at $t=0$ the
first derivative of   $\mathrm{E}(g_t)$ vanishes. This is  the linearised Einstein equation for  the symmetric tensor $h = g^{-1}\dot g$,
which, as we will recall in  Proposition \ref{first}, is equivalent to $\widetilde \Delta_E h = 0$, where  $\widetilde \Delta_E:\Gamma(S^2M) \to \Gamma(S^2M)$ is the
modified Einstein operator, a Laplace type operator defined in  \eqref{D-tilde}. Solutions of the  linearised Einstein equation
are called infinitesimal Einstein deformations. The Ebin slice theorem allows to restrict to directions transversal to the action
of the diffeomorphism group, i.e. to  the space of so-called
essential Einstein deformations denoted by
$$
\scrE(g)  \;= \; \{  h \in \ker \widetilde \Delta_E  \, :  \, \delta^g h = 0  \ \mathrm{and} \ \tr (h)=0 \} .
$$

A curve $g_t$ of metrics is called a second order Einstein deformation if at $t=0$  the first and the second derivative of   
$\mathrm{E}(g_t)$ vanishes. Computing the second derivative at $t=0$ leads to the equation $\mathrm{E}''_g(g^{-1}\dot g, g^{-1}\dot g) + \mathrm{E}'_g(g^{-1}\ddot g) =0$.
An infinitesimal Einstein deformation $h \in \scrE(g)$ is called integrable to second order if there is a second order Einstein deformation
$g_t$ such that $h = g^{-1}\dot g$. Of course if $h$ is not integrable to second order no curve of Einstein 
metrics with
tangent vector $h$ may exist. If all $h \in \scrE(g)$ are not integrable to second order then $g$ is rigid, i.e. the metric $g$ is isolated
in the moduli space of Einstein metrics.  

It is well known that for elements
$h \in \scrE(g)$ the Hessian, i.e. the second derivative,  of the total scalar curvature functional vanishes; 
see \cite[Proposition 4.55]{Besse}. Hence, if $h \in \scrE(g)$ and $g_t$
is a curve as above with $g_0 = g$ Einstein and  $h = g^{-1}\dot g$ then the second derivative of $\mathrm{E}$ can be computed as
\begin{equation}\label{E2}
\left. \frac{\di^2}{\di\!t^2}\right|_{t=0} \mathrm{E}(g_t)  \; = \; \left.  \frac{\di^2}{\di\!t^2}\right|_{t=0} (\ric^{g_t} - Eg_t). 
\end{equation}

Replacing the  symmetric $2$-tensor $\ric^{g_t}$ by the endomorphism  $\Ric^{g_t}$, the linearised Einstein equation can
be written as $\left. \frac{\di}{\di\!t} \right|_{t=0}\Ric^{g_t}=0$ (see Proposition \ref{first}). Moreover, as we will show in Remark \ref{skew-part},
if $g_t$ is a first order Einstein deformation with $h = g^{-1}\dot g$, then the second derivative of $E$ can also  be computed as
$ \mathrm{E}''_g(h,h)  = \left. \tfrac{\di^2}{\di\!t^2}\right|_{t=0}\Ric^{g_t} $, i.e. as the second variation of the Ricci tensor.
Based on this remark a formula for $\mathrm{E}''_g(h,h)$ is given in \cite[Proposition 4.2]{Ko1}. Moreover Koiso shows in \cite[Lemma 4.7]{Ko1} that an element
$h \in \scrE(g)$ is integrable to second order if and only if $\mathrm{E}''_g(h,h)$ is $L^2$-orthogonal to the space $\scrE(g)$ of infinitesimal Einstein deformations. As a necessary 
condition he obtains the vanishing of $\la \mathrm{E}''_g(h,h), h \ra_{L^2}$. 

However, as given by Koiso, the obstruction is very difficult to
calculate for concrete examples. The first major achievement of our article is a new version of  the second variation formula for the Ricci tensor
in Theorem \ref{int-1}.
It is given in an index free fashion using differential operators such as the Fr\"olicher-Nijenhuis bracket $[\cdot, \cdot]^\FN$ and the modified Einstein
operator $\widetilde \Delta_E$ both acting on on symmetric tensors. Instead of translating the terms in the Koiso formula for $\mathrm{E}''_g(h,h)$ we derive the second variation
starting  from the definitions and introducing step by step as many as possible divergence terms which disappear after integration.
This leads to a new and explicit  version of a necessary and sufficient condition for elements in
$ \scrE(g)$ to be integrable to second order formulated in Theorem \ref{int-33}. The main statement is

\begin{teo} \label{main1}
Let $(M,g)$ be a compact Einstein manifold, $\ric^g=Eg$. An element $h \in \scrE(g)$ is integrable to second order if and only if 
$$ -\bfv(h,h)+Eh^2 \perp_{L^2} \scrE(g).
$$
\end{teo}

Here the second order, linear differential operator $\bfv:\Gamma(S^2M) \oplus \Gamma(S^2M) \to \Gamma(S^2M)$ is determined from 
\begin{equation*} 
\la \mathbf{v}(h_1,h_2),h_3\ra_{L^2}=\la \delta^g[h_1,h_2]^{\FN},h_3\ra_{L^2}+\la \delta^g[h_1,h_3]^{\FN},h_2\ra_{L^2}+
\la \delta^g[h_2,h_3]^{\FN},h_1\ra_{L^2}.
\end{equation*}
A remarkable feature in our formulation of the obstruction criterion is that 
the expression
${\bf v}(h_1, h_2, h_3) = \la {\bf v}(h_1, h_2), h_3 \ra_{L^2}$  is symmetric in all three entries. This property makes that in practice the algorithm for computing the obstruction in Theorem \ref{main1} goes as follows
\begin{itemize}
\item[$\bullet$] compute $\bfv(h,h,h)-E\tr_{L^2}h^3$ whenever $h$ in $\scrE(g)$
\item[$\bullet$] polarise the expression found above in order to get $\bfv(h,h,h_1)-E\langle h^2,h_1\rangle_{L^2}$ for all $h,h_1 \in \scrE(g)$; then determine those elements $h \in \scrE(g)$ for which the latter expression vanishes for all $h_1 \in \scrE(g)$. According to Theorem \ref{main1} these elements are thus unobstructed to second order.
\end{itemize}
\begin{rema} \label{rm1-intro}
A necessary condition for having $h \in \scrE(g)$ integrable to second order hence reads 
$\bfv(h,h,h)-E\tr_{L^2}h^3=0$. This recovers and at the same time reformulates in an intrinsic way
Koiso's obstruction  $\la \mathrm{E}''_g(h,h), h \ra_{L^2}$ mentioned above. 
In Section  \ref{koiso} this special case is discussed in more detail. Proposition \ref{div-bra} in that section gives for the first time an index-free expression of the Koiso obstruction.
\end{rema} 
Written as in Theorem \ref{main1} the obstruction to second order integrability is much more amenable  to direct calculation as we will demonstrate in the K\"ahler case.
Recall that if $(M^{2m},g,J)$ is K\"ahler the space of infinitesimal Einstein deformations splits as 
$\scrE(g) = \scrE^+(g) \oplus \scrE^-(g)$,
according to the splitting of the space of symmetric endomorphisms in endomorphisms  commuting respectively anti-commuting with the complex
structure. When $\ric^g=Eg$ and $E<0$ it is well known that $\scrE^{+}(g)=0$ \cite{Ko2}. Due to this and  
as a first application of our criterion from Theorem \ref{main1}, based on simple type considerations, we prove 
\begin{teo} \label{main2}
Assume that $(M^{2m},g,J)$ is a compact  K\"ahler manifold with $\ric^g=Eg$ where $E<0$. Then any infinitesimal 
Einstein deformation in $\scrE(g)$ is integrable to second order.
\end{teo}
Note that no assumption on the complex structure $J$ is needed in Theorem \ref{main2}. As opposed to this, the usual approach to deforming K\"ahler-Einstein metrics consists in first deforming the complex structure--process which is in general obstructed to second order--then applying the Calabi-Yau theorem. See \cite{Besse}, page 350 for more details on the general theory and also \cite{Ko2,Dai} for examples where the obstruction space, 
namely $H^{0,2}(M,T^{1,0}M)$, is non trivial.

Next we treat the case when $E>0$ which is much more complicated due to the fact that space $\scrE^{+}(g)$ consists of eigenforms 
of the Laplacian. However, we are able to give  a necessary condition for the second 
order integrability of elements in $\scrE^{+}(g)$; this turns out to be entirely algebraic in dimension $6$ and also to be 
sufficient under the assumption $\scrE^{-}(g)=0$.

\begin{teo} \label{main4}
Assume that $(M^{2m},g,J)$ is a compact K\"ahler manifold with $\ric^g=Eg$, where $E>0$ and consider $FJ \in \scrE^{+}(g)$. 
Then the following hold
\begin{itemize}
\item[(i)] if $FJ$ is integrable to second order then
\begin{equation*} \label{cond3}
\la \omega \wedge \di\!G, F \wedge \di\!F\ra_{L^2} \,+ \,  \la \omega \wedge \di\!F, F \wedge \di\!G+G \wedge \di\!F \rangle_{L^2} \, = \,  8 \, \la F^2J,G\ra_{L^2}
\end{equation*} 
whenever $GJ \in \scrE^{+}(g)$. In addition, if $\scrE^{-}(g)=0$ then $FJ$ is integrable 
to second order if and only if the last equation holds
\item[(ii)] assume that $m=3$. If $FJ$ is integrable to second order then
\begin{equation*} \label{cond4}
\la F^2J,G \ra_{L^2}=0 \quad \mathrm{for \ all} \ GJ \in \scrE^{+}(g). 
\end{equation*}
In addition, if $\scrE^{-}(g)=0$ then $FJ$ is integrable 
to second order if and only if the last equation holds.
\end{itemize}
\end{teo}

A class of K\"ahler-Einstein manifolds particularly suited for applying this theorem are Hermitian symmetric spaces. However, from  
the classification of Riemannian symmetric spaces of compact type admitting infinitesimal Einstein deformations it follows that 
 the complex Grassmannians are the only Hermitian symmetric spaces  admitting such deformations. In fact
 the complete list can be found in \cite{Ko3} (see also \cite{GaGo}). It consists of the   spaces 
$$
\SU(n), \quad \SU(n)/\SO(n), \quad \SU(2n)/\Sp(n), \quad  \SU(p+q)/\mathrm S(\U(p) \times \U(q)), \quad 
\mathrm{E}_6/\mathrm{F}_4
$$
where $n\ge 3$ and $p \ge q \ge 2$.  For all these spaces it is shown in the references above that $\scrE (g)$ is isomorphic to the Lie algebra of
the isometry group, i.e. to the space of Killing vector fields and in particular that for the complex Grassmannians we have $\scrE^{-}(g)=0$. The proof is a purely representation theoretic 
argument and does not give an explicit isomorphism. However, on all spaces except the complex Grassmannian there exist a parallel symmetric $3$-tensor
(see \cite{Gago2}) which can be used to make the isomorphism explicit (see \cite{Gago2} and \cite{BHMW}).
In  \cite{BHMW} this explicit realisation could be used to calculate the Koiso obstruction in case of the biinvariant metric 
 on  $\SU(n)$  for $n$ odd. It turns out that all infinitesimal  Einstein deformations are obstructed to second order. This was the first 
 example of an irreducible symmetric space admitting non-trivial infinitesimal Einstein deformations, where the symmetric metric is rigid. 
 The only other known example  of a symmetric space where all infinitesimal Einstein deformations are obstructed was $\CM \mathrm{P}^1 \times \CM \mathrm{P}^{2k} $ found by Koiso (see \cite[12.46]{Besse}).
 For all other irreducible symmetric spaces from the list above  nothing on the rigidity question was known so far. 

Below we will study the complex $2$-plane Grassmannian $\mathrm{Gr}_2(\bbC^{n+2}), n\geq 2$ which has the remarkable property to be Hermitian symmetric and
quaternion-K\"ahler at the same time. We are able to provide for the first time a geometric description of all infinitesimal Einstein deformations
which  are unobstructed to second order.
\begin{teo} \label{main5}
Let $M=\mathrm{Gr}_2(\bbC^{n+2}), n\geq 2$ be equipped with its Hermitian symmetric structure $(g,J)$. Then
\begin{itemize}
\item[(i)] the set 
of infinitesimal Einstein deformations in $\scrE(g)$ which are integrable to second order 
is isomorphic to the hyperquadric 
\begin{equation*} \label{hq-i}
\mathscr{C}(n):=\{A \in \su(n+2) : A^2=\frac{1}{2(n+2)}\tr(A^2)\id \} \subseteq \su(n+2)
\end{equation*}
\item[(ii)] when $n$ is odd the metric $g$ is
isolated in the moduli space of Einstein metrics. All its infinitesimal Einstein deformations are obstructed to second order.
\end{itemize}
\end{teo}
The proof of this result mainly hinges on Theorem \ref{main4},(i) which applies since the symmetric metric on the Grassmanian satisfies $\scrE^{-}(g)=0$. However several additional steps pertaining to the specific geometry 
of the Grassmanian are required, as outlined below. As a first  important result  we have a simple  parametrisation of
infinitesimal Einstein deformations by Killing vector fields. This is explicitly given by an isomorphism 
$X \in \mathfrak{aut}(M,g) \mapsto \varepsilon(X) \in \scrE(g)$ which is obtained by taking a suitable 
linear combination of $(\di\!X^{\flat})_0$, which is the primitive piece in $\di\!X^{\flat}$, and of the projection $(\di\!X^{\flat})_Q$; here $Q$ indicates the canonical quaternion-K\"ahler structure on $(M,g)$. Using the isomorphism $\varepsilon$ we fully compute the expression in \eqref{cond3} in terms of polynomial invariants build via algebraic expressions involving 
$X$, $JX$ and certain components of $\di\!X^{\flat}$ where $X \in \aut(M,g)$. These algebraic expressions are further explored and interpreted geometrically using Hamiltonian moment maps, which are closely related to the quaternion-K\"ahler tri-moment map; see Section \ref{hamI} for details. It is this approach that allows proving the isomorphism in Theorem \ref{main5}, (i). The rigidity result in the second part of Theorem \ref{main5} follows from the elementary algebraic observation that $\mathscr{C}(n)=0$ when $n$ is odd.

In light of the rigidity result in Theorem \ref{main5}, (ii) we end with the following 

\begin{conj}The symmetric metric on $\mathrm{Gr}_p(\bbC^{p+q})$ with $p \geq 3$ and $p+q$ odd 
is isolated in the moduli space of Einstein metrics. All its infinitesimal Einstein deformations are obstructed to second order.
\end{conj}
\subsection*{Acknowledgments}
Paul-Andi Nagy was supported by the Institute for Basic Science (IBS-R032-D1). This research was also supported by the Special Priority Program SPP 2026  `Geometry at Infinity'  funded by the DFG. It is a pleasure to thank 
Stuart Hall and Ramiro Lafuente for useful remarks on an earlier version of this paper.

%
\section{Preliminaries} \label{sect2}

\subsection{Differential operators and Weitzenb\"ock formulas} \label{dwz}

Before starting to derive the first and second variation formula for the Ricci tensor we collect a few
definitions of curvature and  differential operators as well as relevant Weitzenb\"ock formulas used throughout the paper.

Let $(M^n, g)$ be a Riemannian manifold and denote with $\nabla^g$ the Levi-Civita connection of $g$ on $TM$ and all
tensor bundles. For the Riemannian curvature tensor $R^g$ we use the convention $R^g(X,Y)=(\nabla^g)^2_{Y,X}-(\nabla^g)^2_{X,Y}$
for tangent vectors $X, Y \in TM$. 
In the following we will make  a notational difference between the Ricci form $\ric^g$ considered as a symmetric bilinear form and the Ricci tensor $\Ric^g$  
considered as a symmetric endomorphism. We will denote by $g^{-1}$ the isomorphism, induced by the metric, between symmetric bilinear forms and symmetric
endomorphisms, e.g. we have $g^{-1} \ric^g = \Ric^g$. The subbundle of $g$-symmetric endomorphisms   is denoted by $S^2M \subseteq \End(TM)$. It is preserved 
by $\nabla^g$ and  the curvature action $h \mapsto \ring{R}h:=\sum_i R^g(e_i,\cdot)he_i$, where $\{ e_i\}$ is some $g$-orthonormal frame. 

We will use the coupled exterior differential $\di_{\nabla^g}:\Omega^{k}(M,TM) \to \Omega^{k+1}(M,TM)$. For $k=0$, i.e. on vector fields,
it coincides with the covariant derivative $\nabla^g$. For $k=1$, i.e. on endo\-morphisms $h$ considered as elements in $ \Omega^1(M,TM)$, it  
can be written as
$\di_{\nabla^g}h(X,Y)=(\nabla^g_Xh)Y-(\nabla_Y^gh)X$. The operator formally  adjoint to $\di_{\nabla^g}$  is the divergence operator  $\delta^g:\Omega^{k}(M,TM) \to \Omega^{k - 1}(M,TM)$ defined according to the convention $\delta^g = -  \sum_i e_i \lrcorner \nabla^g_{e_i}$. 
%

%
%

Restricting the divergence to symmetric $2$-tensors we have $\delta^g : \Gamma(S^2M) \to  \Omega^1(M)$. Its formal adjoint 
$\delta^{\star_g} : \Omega^1(M) \to \Gamma(S^2M)$ is the symmetric part of $\nabla^g$, i.e.  $\nabla^g=\delta^{\star_g}+\frac{1}{2}\di$ on $\Omega^1(M)$.
As a consequence we note for later application that $\tr (\delta^{\star_g} \alpha)  = - \di^{\star}\!\alpha$ holds for any $1$-form $\alpha$.
We also have the well-known formula  $\delta^{\star_g}\alpha =\frac{1}{2}g^{-1}\L_{\alpha^{\sharp}}g$, characterising the kernel of $\delta^{\star_g}$
as $1$-forms dual to Killing vector fields.

In dealing with the trace and divergence of symmetric tensors we will use the Bianchi operator $\D^g: \Gamma(S^2M) \to \Omega^1(M)$ given by  $\D^g=2\delta^g+\di\! \circ \tr$. 
Certainly $\D^g$ vanishes on the space of the so-called $\TT$-tensors defined by
$$
\TT(g):=\{h \in \Gamma(S^2M) : \delta^g h=0  \ \mathrm{and} \ \tr(h)=0\} . 
$$
However terms involving divergence and trace must be kept track of in order to obtain gauge invariant equations. 
For latter use recall that $\D^g \, \Ric^g =0$, which is a consequence of the differential Bianchi identity 
(see \cite{Besse}, 12.33).

On symmetric $2$-tensors and for Einstein metrics $g$ with Einstein constant $E$ we will consider the Einstein operator 
$\Delta_E := (\nabla^g)^{\star}\nabla^g-2\ring{R} $. Note that $\Delta_E = \Delta_L - 2E$, where $\Delta_L$ is the Lichnerowicz
Laplacian on symmetric $2$-tensors (see \cite{Besse}, 1.143). We will also use a modification of the Einstein operator, 
the differential operator $\widetilde{\Delta}_E:
\Gamma(S^2M) \to \Gamma(S^2M)$ given by 
\begin{equation}\label{D-tilde}
 \widetilde{\Delta}_E \;:= \;  \Delta_E \, - \, 2\delta^{\star_g} \circ (\delta^g+\frac{1}{2}\di \circ \tr)\; = \; \Delta_E \, - \, \delta^{\star_g} \circ \D^g.
\end{equation}

The following Weitzenb\"ock formula on  $2$-tensors  
(see \cite{bourg}, Proposition 4.1) will be used at several  places in this paper.
\begin{equation}\label{wz0}
\delta^g \circ \di_{\nabla^g} \, + \, \di_{\nabla^g} \circ \delta^g \; =\; (\nabla^g)^{\star}\nabla^g \, + \, E \,  - \, \ring{R} \;= \;\Delta_E\,  + E \,  +\,  \ring{R}  . 
\end{equation}
Later we will apply it in the form
\begin{equation} \label{wz1}
\delta^g \circ \di_{\nabla^g} \;=\;  -\nabla^g \circ \delta^g  \, + \, \Delta_E \, + E+\,  \ring{R} \; =\;  - (\delta^{\star_g}+\frac{1}{2}\di) \circ \delta^g\, +\, \Delta_E\, +E+\, \ring{R} .
\end{equation}
We still need another Weitzenb\"ock formula, this time on $1$-forms. Here we have
\begin{equation} \label{wz2}
2 \delta^g \delta^{\star_g}  \, - \,   \di \di^{\star}  \; = \;  \Delta \, - \,  2E ,
\end{equation}
where $\Delta = \di \di^{\star} + \di^{\star} \di$ is the Hodge Laplacian and $g$ is again an Einstein metric. Note that the formula  is a special case of a 
Weitzenb\"ock formula on symmetric tensors of arbitrary degree  (see \cite{AU}, proof of Proposition 6.2, for further details).

\subsection{The Fr\"olicher-Nijenhuis bracket} \label{FN-br}

On a given manifold $M$ we first recall that the Fr\"olicher-Nijenhuis bracket  for sections $h$ of $\End(TM)$ reads 
$$
[h,h]^{\FN}(X,Y) \, = \, -h^2[X,Y] \, + \, h([hX,Y] \, + \, [X,hY]) \, - \, [hX,hY]  .
$$  
From now on assume that $g$ is a Riemannian metric on $M$. Recast in terms of the Levi-Civita connection $\nabla^g$ of $g$ this bracket reads 
\begin{equation} \label{bra-na}
[h,h]^{\FN}(X,Y) \, = \, -(\nabla^g_{hX}h)Y \, + \, (\nabla^g_{hY}h)X  \, + \, h ((\di_{\nabla^g}h)(X,Y)) .
\end{equation}
Here  $\di_{\nabla^g}$ acts on $h$  considered as  an element of $\Omega^1(M, TM)$.
More concisely 
\begin{equation} \label{bra-nac}
[h,h]^{\FN} \, = \, -h\sharp \di_{\nabla^g}h \, + \, \di_{\nabla^g}h^2
\end{equation}
where the algebraic action $\alpha \in \Lambda^2(M,TM) \mapsto h \sharp \alpha \in \Lambda^2(M,TM)$ is defined according to $h \sharp \alpha(X,Y)=\alpha(hX,Y)+\alpha(X,hY)$.
The Fr\"olicher-Nijenhuis bracket is extended to a symmetric bracket on the space  $\Gamma(S^2M )$ via 
\begin{equation} \label{bra-nac2}
2[h_1,h_2]^{\FN} \, = \, -(h_1 \sharp \di_{\nabla}h_2 \, + \, h_2 \sharp \di_{\nabla}h_1) \, + \, \di_{\nabla} \{h_1,h_2\}
\end{equation}
where the anti-commutator $\{h_1,h_2\}:=h_1 \circ h_2+h_2 \circ h_1$. In particular we have 
$[\id ,h]^{\FN}=0$ for all $h \in \End(TM)$.

%
\section{The second variation of the Ricci tensor}
%

\subsection{Derivation from first principles}

Let $(M^n,g)$ be a  Riemannian manifold and consider metrics of the form $g_h=g(h \cdot, \cdot)$ where $h$ is a symmetric and positive definite
endomorphism. Indicate the Levi-Civita connection  of $g_h$ with $\nabla^h$.  Then a direct computation based on  Koszul's formula shows that
\begin{equation*}
\nabla^h_XY \, = \, \nabla^g_XY \, + \, \frac{1}{2}\eta_XY
\end{equation*}
where the tensor $\eta :TM \to \End(TM)$, denoted with $X \mapsto \eta_X$, is explicitly determined from 
\begin{equation} \label{eta1}
g_h(\eta_XY,Z) \, = \, g((\nabla^g_Xh)Y,Z) \, + \, g(X,(\nabla^g_Yh)Z \, - \, (\nabla^g_Zh)Y).
\end{equation}
Because of  $\eta_XY=\eta_YX$  the above formula can also be proved by checking that $\nabla^h$ preserves $g_h$ and is torsion free. The 
Riemann curvature tensor of $g_h$ is determined from 
\begin{equation} \label{curv-comp}
R^h(X,Y)Z \, = \, R^g(X,Y)Z \, - \, \frac{1}{2}\bigl ((\nabla^g_X\eta)_YZ-(\nabla^g_Y\eta)_XZ \bigr ) \, - \,\frac{1}{4}[\eta_X,\eta_Y]Z.
\end{equation} 
Now consider a family $g_t = g_{h_t}$ of Einstein metrics, for small $t$, i.e $\vert t \vert <\varepsilon$ such that 
$g_0=g$ (hence $h_0=\id$) and denote with $R^{g_t}$ the Riemann curvature tensor of $g_t$. The Einstein equation for the family $g_t$ reads 
$R^{g_t}(f_i,X)f_i=EX$, where $\{f_i\}$ is an orthonormal basis w.r.t. $g_t$.  Here and in the rest of the paper we will use Einstein's sum convention.
Note that the Einstein constant along the curve $g_t$ is constant and 
equal to $E$, the Einstein constant of $g=g_0$. Indeed, since Einstein metrics are critical points of the total scalar curvature functional this
functional is constant along a curve of Einstein metrics. But the scalar curvature of an Einstein metric is constant and the volume is 
normalised to one, thus the claim follows. 
We have
$f_i=h_t^{-\frac{1}{2}}e_i$, where $\{e_i\}$ is a $g$-orthonormal basis. Since $h_t$ is $g$-symmetric it follows that  
\begin{equation} \label{eins-main}
\Ric^{g_t}X \, = \, R^{g_t} (e_i,X)h_t^{-1}e_i .
\end{equation}
Remark that the last equation is basis independent. The aim is to differentiate this equation up to second order, 
at $t=0$. The first part of this process is fairly formal and unfolds  as follows. Start from 
\begin{equation}\label{der-0}
\left.{\frac{\di h_{t}^{-1}}{\di\!t}}\right|_{t=0} = \, -\dot{h}, \qquad   \left.{\frac{\di^2 h_{t}^{-1}}{\di\!t^2}}\right|_{t=0}  = \, 2\dot{h}^2-\ddot{h}
\end{equation} 
where $\dot{h} = g^{-1}\dot g$ and $\ddot{h}$ denote the derivatives of $h_t$ at $t=0$. These equations follow from taking
derivatives  in $t$ of the trivial identity $h_t \circ h_t^{-1} = \id$. For simplifying the notation we will always write $\dot h^2$ for $(\dot h)^2$.
Denoting by $\dot{R}$ respectively $\ddot{R}$ the derivatives of $R^{g_t}$ at $t=0$, differentiation in \eqref{eins-main} together with \eqref{der-0}
shows that 
\begin{equation} \label{der-12}
(\left. \frac{\di}{\di\!t}\right|_{t=0}\Ric^{g_t})X \,=\, \dot{R}(e_i,X)e_i-R^g(e_i,X)\dot{h}e_i
\end{equation}
respectively
\begin{equation} \label{der-13}
(\left. \frac{\di^2}{\di\!t^2}\right|_{t=0}\Ric^{g_t})X \,= \, \ddot{R}(e_i,X)e_i \, - \, 2\dot{R}(e_i,X)\dot{h}e_i \, + \, R^g(e_i,X)(-\ddot{h} + 2 \dot{h}^2)e_i.
\end{equation} 
To obtain an explicit form for the expression above we need to compute the curvature derivatives
at $t=0$. Taking into account that $\eta=0$ at $t=0$, differentiation in \eqref{curv-comp} shows that 
\begin{equation} \label{curv-2}
\begin{split}
\dot{R}(X,Y)Z \,=\, &  -\frac{1}{2} \big(   (\nabla^g_X\dot{\eta})_YZ-(\nabla^g_Y\dot{\eta})_XZ \big)\\
\ddot{R}(X,Y)Z \, =  \, & -\frac{1}{2}\big(  (\nabla^g_X\ddot{\eta})_YZ-(\nabla^g_Y\ddot{\eta})_XZ\big) - \frac{1}{2}[\dot{\eta}_X, \dot{\eta}_Y]Z
\end{split}
\end{equation}
Here $\dot{\eta}$ and $\ddot{\eta}$ are the derivatives of $\eta$ at $t=0$. Finally, we need to record the explicit formulas determining 
$\dot{\eta}$ and $\ddot{\eta}$. Differentiating similarly in \eqref{eta1} we get
\begin{equation} \label{eta-2}
\begin{split}
g(\dot{\eta}_XY,Z) \, = \, & \, g((\nabla^g_X\dot{h})Y,Z) \, + \, g(X,(\nabla^g_Y\dot{h})Z  -  (\nabla^g_Z\dot{h})Y)\\[.8ex]
g(\ddot{\eta}_XY+2\dot{h}(\dot{\eta}_XY),Z) \, = \, & \, g((\nabla^g_X\ddot{h})Y,Z) \, + \, g(X,(\nabla^g_Y\ddot{h})Z -  (\nabla^g_Z\ddot{h})Y).
\end{split}
\end{equation}

Our first objective is the computation of $\ddot{R}(e_i,X)e_i$, using the second equation of  \eqref{curv-2} together with the second equation of   \eqref{eta-2}. In the following preparatory lemma 
the divergence  $  \delta^g\dot{\eta}  $ (and similarly that of $\ddot \eta $)
is defined according to 
$$
(\delta^g\dot{\eta})_X = - (\nabla^g_{e_i } \dot \eta)_{e_i}X 
= - \nabla_{e_i}^g (\dot \eta_{e_i} X)  + \dot \eta_{e_i} (\nabla_{e_i}^g X)+\dot \eta_{\nabla_{e_i}^ge_i}X = - (\nabla_{e_i}^g \dot \eta)_X e_i.
$$
Note that in the last equation we took into account the symmetry of $\dot \eta$.
\begin{lema} \label{traces}For any $\dot h \in \Gamma(S^2M)$ the following hold
\begin{itemize}
\item[(i)] \;
$\dot{\eta}_{e_i}e_i \, = \, -\D^g\dot{h}$
\item[(ii)] \;
$ \ddot{\eta}_{e_i}e_i \, = \, -\D^g\ddot{h}  \, +  \,  2 \dot{h}\D^g \dot{h}$
\item[(iii)] \;
$\delta^g \dot{\eta} \; \, = \; (\nabla^g)^{\star}\nabla^g\dot{h}  \, + \,\di\!\delta^g\dot{h}$.
\end{itemize}
\end{lema}
\begin{proof}
(i) Since $\dot{h}$ is $g$-symmetric, the claim follows from the first equation in \eqref{eta-2}.\\
(ii) follows from the second equation in \eqref{eta-2}, (i) and the fact that $\ddot{h}$ is $g$-symmetric.\\
(iii) We write $\nabla=\nabla^g$ during this proof, in order to simplify notation. Taking the covariant derivative $\nabla_{e_i}$  in 
$$ g(\dot{\eta}_{e_i}X,Y)=g((\nabla_{e_i}\dot{h})X,Y)+g((\nabla_{X}\dot{h})e_i,Y)-
g((\nabla_{Y}\dot{h})e_i,X)
$$
leads to $g((\delta^g\dot{\eta})_X,Y)=g((\nabla^{\star}\nabla \dh)X,Y)-
g((\nabla^2_{e_i,X}\dot{h})e_i,Y)+g((\nabla^2_{e_i,Y}\dot{h})e_i,X)$. From the Ricci identity we get $(\nabla^2_{e_i,X}\dot{h})e_i=(\nabla^2_{X,e_i}\dot{h})e_i-[R^g(e_i,X), \dot h]e_i=
-\nabla_X(\delta^g\dot{h})-(\ring{R}^g\dot{h}-E\dot{h})_X$. As the tensor $\ring{R}^g\dot{h}-E\dot{h}$ is symmetric the claim follows.
\end{proof}

\subsection{The first variation of the Ricci tensor}
The operator $\widetilde{\Delta}_E$ defined in \eqref{D-tilde}
essentially describes the first variation of the Ricci tensor, see \cite[(1.180a)]{Besse}. As a preparation for dealing with higher order derivatives we quickly re-confirm this below. 
\begin{pro} \label{first}
Let $(M,g)$ be an  Einstein manifold with $\ric^g=Eg$ and assume that $g_t$ is a smooth curve of metrics with $g_0=g$ and  $\dot{h}:=g^{-1}\dot{g}$. Then 
$$ 
  \left.\frac{\di}{\di\!t}\right|_{t=0}\Ric^{g_t}  \, =  \, \left. g^{-1}\frac{\di}{\di\!t }\right|_{t=0}(\ric^{g_t}-Eg_t) \, = \, \frac{1}{2}\widetilde{\Delta}_E \dot{h} .
$$
\end{pro}
\begin{proof}
By \eqref{curv-2} and the symmetry of $\dot{\eta}$ in its first two slots, i.e. since $\dot \eta_XY = \dot \eta_YX$, we have 
\begin{equation*}
\dot{R}(e_i,X)e_i=\frac{1}{2}(\delta^g\dot{\eta})_X+\frac{1}{2}(\nabla^g_X \dot{\eta})_{e_i}e_i.
\end{equation*}
After using Lemma \ref{traces}, (i) and (iii) this leads to 
\begin{equation*}
\begin{split}
\left.(\frac{\di}{\di\!t}\right|_{t=0}\Ric^{g_t})_X &= \dot{R}(e_i,X)e_i-(\ring{R}\dot{h})_X=\frac{1}{2}(\nabla^{\star}\nabla \dot{h}+\di\! \delta^g \dot{h})_X-
\frac{1}{2}((\delta^{\star}+\frac{1}{2}\di)\D^g\dot{h})_X-(\ring{R}\dot{h})_X \\
&=  \frac{1}{2}(\widetilde{\Delta}_E \dot{h})_X
\end{split}
\end{equation*}
since $\di \circ (\delta^g-\frac{1}{2} \D^g)=-\frac{1}{2}\di^2 \circ \tr=0$. The claim is then fully proved since
$$\left.\frac{\di}{\di\!t}\right|_{t=0}\Ric^{g_t}  \, =\left.\frac{\di}{\di\!t}\right|_{t=0}\,
(g_t^{-1} \ric^{g_t}) =
\left ( \left.\frac{\di}{\di\!t}\right|_{t=0}g_t^{-1} \right )\ric^{g}
+ g^{-1} 
\left.\frac{\di}{\di\!t}\right|_{t=0}\ric^{g_t} =  g^{-1}
\left.\frac{\di}{\di\!t}\right|_{t=0}(\ric^{g_t}-Eg_t) .
$$
\end{proof}

The first variation formula of the Ricci tensor given in Proposition \ref{first} shows that $\ker \widetilde{\Delta}_E$ describes first order 
Einstein deformations. Recall that $\scrE(g) \subset \ker \widetilde{\Delta}_E$ by definition. Moreover, Proposition \ref{first} implies via a gauge invariance 
argument the identity 
\begin{equation} \label{ga-ba}
\widetilde{\Delta}_E \circ \delta^{\star_g}=0 .
\end{equation}
The identity will be used later on in the paper and is proved  as follows. Let $X$ be a vector field on 
$M$ with flow $\Phi_t$ and consider the special curve of Einstein metrics  defined by $g_t=\Phi_t^{\star}g$. It starts at $g$ and 
satisfies $g^{-1} \dot g=g^{-1}\L_Xg=2\delta^{\star_g}X$. 
Thus \eqref{ga-ba} follows from the first variation formula for the Ricci tensor of $g_t$  with $\dot h = 2\delta^{\star_g}X$. 
\subsection{The second variation of the Ricci tensor - first steps}
We want to calculate the second variation of the Ricci tensor by applying \eqref{der-13}.
First we observe that the two second order summands in \eqref{der-13} are computed as  follows.

\begin{lema} \label{step2}
We have 
\begin{equation*}
\begin{split}
\ddot{R}(e_i,X)e_i \,  + \, R^g(e_i,X)(-\ddot{h}+2{\dot{h}}^2)e_i 
& \;=\; \frac{1}{2}(\widetilde{\Delta}_E\ddot{h})_X \,+ \, \frac{1}{2}X \lrcorner \di(\delta^g-\D^g)\ddot{h}  \, + \, 2\ring{R}^g(\dot{h}^2)X 
 \\
 &
\qquad  \qquad \qquad  \, - \, \frac{1}{2}(\dot{\eta}_{e_i})^2X \, - \, \frac{1}{2}\dot{\eta}_{\D^g\dot{h}}X-
\nabla^g_X(\dot{h}\D^g\dot{h})
\\ &
\qquad  \qquad \qquad  
\;-\delta^g(U \mapsto \dot{h} \circ \dot{\eta}_U)_X
\end{split}
\end{equation*}
for all $X \in TM$.
\end{lema}
\begin{proof}
From the second equation in \eqref{curv-2} we obtain 
$$ \ddot{R}(e_i,X)e_i=\frac{1}{2}(\delta^g\ddot{\eta})_X+\frac{1}{2}\nabla^g_X (\ddot{\eta}_{e_i}e_i)-\frac{1}{2}(\dot{\eta}_{e_i})^2X-\frac{1}{2}\dot{\eta}_{\D^g\dot{h}}X
$$
after taking into account that $\ddot{\eta}$ and $\dot{\eta}$ are symmetric in the first two slots, together with the expression for $\dot{\eta}_{e_i}e_i$ from Lemma \ref{traces}, (i). At the same time, a computation entirely similar to that used for proving (iii) in Lemma \ref{traces} shows that 
\begin{equation*}
\delta^g\ddot{\eta}=-2\delta^g(U \mapsto \dot{h} \circ  \dot{\eta}_U)+(\nabla^g)^{\star}\nabla^g \ddot{h}+\di\!\delta^g \ddot{h}
\end{equation*}
The claim follows  using the expression for $\delta^g\dot{h}$ and $\ddot{\eta}_{e_i}e_i$ from Lemma \ref{traces}.
\end{proof}

Next, the remaining summand of   \eqref{der-13} is computed using the notation  $A:=\dot{\eta}_{e_i}\dot{h}e_i$. As we will see later
 the action of the skew-symmetric endomorphism $R^-(\dot h)$ below can be ignored.
\begin{lema} \label{step3}
We have 
\begin{equation*}
-2\dot{R}(e_i,X)\dot{h}e_i=-\delta^g(U \mapsto \nabla^g_{\dot{h}U}\dot{h})_X+\dot{\eta}_{\delta^g\dot{h}}X
+\dot{\eta}_{e_i}(\nabla^g_X\dot{h})e_i-\nabla^g_XA+ R^-(\dot h)X 
\end{equation*}
for some skew-symmetric endomorphism $R^-(\dot{h})$ of $TM$.
\end{lema}
\begin{proof}Using the first equation of \eqref{curv-2} we compute 
\begin{equation*}
\begin{split} 
-2\dot{R}(e_i,X)\dot{h}e_i \; =&  \quad(\nabla^g_{e_i}\dot{\eta})_{\dot{h}e_i}X \, - \, (\nabla^g_X \dot{\eta})_{e_i}\dot{h}e_i\\
=&-\delta^g(U \mapsto \dot{\eta}_{\dot{h}U})_X \, + \, \dot{\eta}_{\delta^g\dot{h}}X \, - \, \nabla_X^gA \, + \, \dot{\eta}_{e_i}
(\nabla_X^g\dot{h})e_i.
\end{split}
\end{equation*}
The claim follows by observing, via \eqref{eta-2}, that the symmetric component of $\dot{\eta}_{\dot{h}U}$ equals 
$\nabla_{\dot{h}U}^g\dot{h}$.
\end{proof}
\begin{rema}
The vector field $A$ can also be computed directly from \eqref{eta-2} and reads 
\begin{equation} \label{comp-A}
A=-2\delta^g\dot{h}^2-\frac{1}{2}\di\!\tr(\dot{h}^2)+2\dot{h}\delta^g \dot{h}.
\end{equation}
\end{rema}
In order to be able to combine Lemmas \ref{step2} and \ref{step3} into a preliminary description of the second variation of the Ricci tensor the following observation plays a major role.
\begin{lema} \label{dir-1}
The tensors $\dot{h}$ and $\dot{\eta}$ satisfy 
\begin{equation*}
g(\nabla^g_{\dot{h}X}\dot{h}+\dot{h} \circ \dot{\eta}_X,H)=g(-X \lrcorner [\dot{h},\dot{h}]^{\FN}+\nabla^g_X \dot{h}^2,H)
\end{equation*}
whenever $X \in TM$ and $H \in \Gamma(S^2M)$.
\end{lema}
\begin{proof}This is a straightforward tensorial computation. From \eqref{eta1} we get 
\begin{equation*}
\begin{split}
g(( & \nabla^g_{\dot{h}  X}\dot{h})Y  + \, \dot{h} (\dot{\eta}_XY),Z) \,=\,g((\nabla^g_{\dot{h}X}\dot{h})Z-(\nabla^g_{\dot{h}Z}\dot{h})X,Y) \,+\, g((\nabla^g_{X}\dot{h})Y+(\nabla^g_{Y}\dot{h})X,\dot{h}Z)\\
=&-g([\dot{h},\dot{h}]^{\FN}(X,Z),Y) \, + \, g(\dot{h}\left ( (\nabla^g_{X}\dot{h})Z-(\nabla^g_{Z}\dot{h})X\right ),Y)
+ \, g((\nabla^g_{X}\dot{h})Y  +  (\nabla^g_{Y}\dot{h})X,\dot{h}Z)\\
=&
-g([\dot{h},\dot{h}]^{\FN}(X,Z),Y) \, + \, g((\nabla_X^{g}\dot{h}^2)Y,Z)
\, + \, g(X,(\nabla^g_Y \dot{h})\dot{h}Z-
(\nabla^g_Z \dot{h})\dot{h}Y)
\end{split}
\end{equation*}
by using successively that the tensor $\nabla^{g}\dot{h}$ is symmetric in the last two slots, the expression of the Fr\"olicher-Nijenhuis bracket in \eqref{bra-na} and that $\dot{h}$ is symmetric. As the last summand in the equation displayed above is skew symmetric in $Y$ and $Z$, the claim follows.
\end{proof}

We summarise the expression for the second derivative of the Ricci tensor obtained so far. 
\begin{pro} \label{prel-fo}
Let $(M,g)$ be an  Einstein manifold with $\ric^g=Eg$ and let $g_t$ be a smooth curve with $g_0=g$. Then for any symmetric endomorphism $H \in \End (TM)$
we have
\begin{equation*}
\begin{split}
\left. g(\frac{\di^2}{\di\!t^2} \right|_{t=0}
\Ric^{g_t},H)&\;=\;g(\frac{1}{2}\widetilde{\Delta}_E\ddot{h}-\Delta_E\dot{h}^2+\delta^g[\dot{h},\dot{h}]^{\FN},H)\\
&\qquad \qquad -  g(\delta^{\star}(A+\dot{h}\D^g\dot{h}), H) \,-\, \frac{1}{2}g(\nabla^g_{\grad (\tr(\dot{h}))}\dot{h},H)\\[.5ex]
&\qquad \qquad  -\frac{1}{2}g( \dot{\eta}_{e_i}^2,H) \,+\, g(\dot{\eta}_{e_i}(\nabla^g_{e_j}\dot{h})e_i,He_j) ,
\end{split}
\end{equation*}
where the symmetric tensors 
$\dot{h}:=g^{-1}\dot{g}$ and $\ddot{h}:=g^{-1}\ddot{g}$. 
\end{pro}
\begin{proof}
Substituting the expressions from  Lemmas \ref{step2} and \ref{step3} into \eqref{der-13} we find that 
\begin{equation*}
\begin{split}
( \left.\frac{\di^2}{\di\!t^2}\right|_{t=0}
\Ric^{g_t})X \;=\;  &\frac{1}{2}(\widetilde{\Delta}_E\ddot{h})_X - \nabla^g_X
(A+\dot{h}\D^g\dot{h})+\dot{\eta}_{\delta^g\dot{h} -\frac{1}{2}\D^g\dot{h}}  X\\
& \qquad  -\delta^g(U \mapsto \nabla^g_{\dot{h}U}\dot{h}+\dot{h} \circ \dot{\eta}_U)_X+2(\ring{R}\dot{h}^2)X\\
& \qquad -\frac{1}{2}\dot{\eta}_{e_i}^2X+\dot{\eta}_{e_i}(\nabla^g_X \dot{h})e_i + R^-(\dot h)X.
\end{split}
\end{equation*}
From Lemma \ref{dir-1} we see that the  scalar product with $H$ of the second line above gives
\begin{equation*}
\begin{split}
g( \nabla^g_{e_i} &(\nabla^g_{\dot h e_i} \dot h + \dot h \circ \dot \eta_{e_i}) + 2\ring{R}\dot{h}^2 , H)  \;= \;
g(\nabla^g_{e_i} ( - e_i \lrcorner  [\dot h, \dot h]^{\FN} + \nabla_{e_i} \dot h^2) + 2\ring{R}\dot{h}^2 , H)
\\[.8ex]
&
\qquad
=\;  g(\delta^g  [\dot h, \dot h]^{\FN} -  (\nabla^g)^{\star}\nabla^g \dot h^2  +   2\ring{R}\dot{h}^2 , H)
\;=\; g(\delta^g  [\dot h, \dot h]^{\FN} -  \Delta_E \dot h^2 , H) .
\end{split}
\end{equation*}
Since $\nabla^gX=\delta^{\star}X+\frac{1}{2}\di X$  for a vector field $X$, 
the symmetric part of $\nabla^gX$ is $\delta^{\star}X$. Moreover,  from the  first equation of \eqref{eta-2} it follows  that the symmetric 
part of $\dot \eta_X$ is  $\nabla^g _X \dot h$. To conclude, we recall the definition $\D^g = 2 \delta^g + \di \circ \tr$.
%
\end{proof}

\begin{rema} \label{skew-part}
The symmetry of the Ricci form $\ric^{g_t}$ entails that the $g$-skew-symmetric component in 
$2\left.\frac{\di^2}{\di\!t^2}\right|_{t=0}
\Ric^{g_t}$ is automatically determined in terms of lower 
order derivatives of $\Ric^{g_t}$ and thus is irrelevant when $g_t$ is tangent to a first order Einstein deformation. This follows from the identity 
\begin{equation*} \label{rvsR}
g^{-1} 
\left.\frac{\di^2}{\di\!t^2}\right|_{t=0}
(\ric^{g_t}-Eg_t) \; = \; 
\left.\frac{\di^2}{\di\!t^2}\right|_{t=0}
\Ric^{g_t} \, + \,2(g^{-1}\dot{g}) \circ 
\left.\frac{\di}{\di\!t}\right|_{t=0}
\Ric^{g_t}.
\end{equation*}
\end{rema}


%
\subsection{The quadratic terms} \label{quad-t}
%
In order to bring the computation of the second variation of the Ricci tensor to final form further insight into the algebraic structure of the quadratic terms 
involving $\dot{\eta}$ and $\nabla^g\dot{h}$ in the last line of Proposition \ref{prel-fo} is needed. The crucial technical observation in that 
direction is the following 
\begin{lema} \label{prod-1}
Assume that $H \in \Gamma(S^2M)$ and let $L:= \nabla_{e_i}^g\dot{h} \circ \nabla_{e_i}^g\dot{h}$ in 
$\Gamma(S^2M)$. We have 
\begin{itemize}
\item[(i)] $g(\dot{\eta}_{e_i}(\nabla^g_{e_j}\dot{h})e_i,He_j)=-g(\di_{\nabla^g}\!\dot{h}, H \sharp \di_{\nabla^g}\!\dot{h})+g(L,H)$
\\[-1ex]
\item[(ii)] $g(\dot{\eta}_{e_i}^2,H)=g(L,H)-g(\di_{\nabla^g}\dot{h},H\sharp \di_{\nabla^g}^g\dot{h}).$
\end{itemize}
In particular we have
\begin{equation} \label{combo}
-\frac{1}{2}g( \dot{\eta}_{e_i}^2,H) \,+ \, g(\dot{\eta}_{e_i}(\nabla^g_{e_j}\dot{h})e_i,He_j) \,=\,
\frac{1}{2} \left (g(L,H) \, - \, g(\di_{\nabla^g}\dot{h},H\sharp \di_{\nabla^g}\dot{h}) \right ).
\end{equation}
\end{lema}
\begin{proof}
(i) Expanding $\dot{\eta}$ according to \eqref{eta-2} shows that 
\begin{equation*}
\begin{split}
g(\dot{\eta}_{e_i}(\nabla^g_{e_j}\dot{h}) & e_i,He_j) =g((\nabla^g_{e_j}\dot{h})e_i,e_k)g(\dot{\eta}_{e_i}e_k,He_j)\\
=& \, g((\nabla^g_{e_j}\dot{h})e_i,e_k)\biggl(g((\nabla^g_{e_i}\dot{h})e_k,He_j) \, + \, g(e_i,
(\nabla^g_{e_k}\dot{h})He_j \, - \,(\nabla^g_{He_j}\dot{h})e_k) \biggr )\\
=& \; g((\nabla^g_{e_i}\dot{h})He_j,(\nabla^g_{e_j}\dot{h})e_i) \, + \, g((\nabla^g_{e_j}\dot{h})e_k,(\nabla^g_{e_k}\dot{h})He_j) \, -\, g((\nabla^g_{e_j}\dot{h})e_i,(\nabla^g_{He_j}\dot{h})e_i)\\
=&\; 2g((\nabla^g_{e_i}\dot{h})He_j,(\nabla^g_{e_j}\dot{h})e_i)  \, - \, g(\nabla^g_{e_j}\dot{h},\nabla^g_{He_j}\dot{h})
\end{split}
\end{equation*}
after also taking into account that $\dot{h}$ and $H$ are symmetric. At the same time we have, directly from the definitions,
\begin{equation*}
\begin{split}
g(\di_{\nabla^g}\!\dot{h}, & H \sharp \di_{\nabla^g}\!\dot{h})=\frac{1}{2}
g(\di_{\nabla^g}\!\dot{h}(e_i,e_j), \di_{\nabla^g}\!\dot{h}(He_i,e_j)+\di_{\nabla^g}\!\dot{h}(e_i,He_j))\\
=&
g(\di_{\nabla^g}\!\dot{h}(e_i,e_j), \di_{\nabla^g}\!\dot{h}(He_i,e_j))
=g((\nabla^g_{e_i}\dot{h})e_j-(\nabla_{e_j}^g\dot{h})e_i,(\nabla^g_{He_i}\dot{h})e_j-(\nabla^g_{e_j}\dot{h})He_i)\\
=&g(\nabla^g_{e_i}\dot{h},\nabla^g_{He_i}\dot{h})-g((\nabla^g_{e_j}\dot{h})e_i,(\nabla^g_{He_i}\dot{h})e_j)
-g((\nabla^g_{e_i}\dot{h})e_j,(\nabla_{e_j}^g\dot{h})He_i)+g(L,H)\\
=&g(\nabla^g_{e_i}\dot{h},\nabla^g_{He_i}\dot{h})-2g((\nabla^g_{e_i}\dot{h})e_j,(\nabla^g_{e_j}\dot{h})He_i)+
g(L,H)
\end{split}
\end{equation*} 
since $H$ is symmetric w.r.t. $g$. The claim follows by comparison.\\
\noindent
(ii) We have $\dot{\eta}_X=\nabla_X^g\dot{h}+A_X$ where  the endomorphism $A_X$ is $g$-skew symmetric 
 and satisfies $g(A_XY,Z)=g(X,\di_{\nabla^g}\dot{h}(Y,Z))$. Since $H$ is symmetric it follows that 
$$g(\dot{\eta}_{e_i}^2,H)=g((\nabla_{e_i}^g\dot{h})^2,H)+g(A_{e_i}^2,H).$$
But
\begin{equation*}
\begin{split} 
g(A_{e_i}^2,H) & =g(A_{e_i}^2e_j,He_j)=-g(A_{e_i}e_j,A_{e_i}He_j)=-
g(A_{e_i}e_j,e_k)g(A_{e_i}He_j,e_k)\\
=&-g(e_i,\di_{\nabla^g}\dot{h}(e_j,e_k))g(e_i,\di_{\nabla^g}\dot{h}(He_j,e_k))=-g(\di_{\nabla^g}\dot{h}(e_j,e_k), \di_{\nabla^g}\dot{h}(He_j,e_k))\\
=&-\frac{1}{2}g(\di_{\nabla^g}\dot{h}(e_j,e_k),(H\sharp \di_{\nabla^g}\dot{h})(e_j,e_k))=
-g(\di_{\nabla^g}\dot{h},H\sharp \di_{\nabla^g}\dot{h}).
\end{split}
\end{equation*}
This finishes the proof of the claim.
\end{proof}
At this stage further interpretation of the r.h.s of \eqref{combo} in terms of the Fr\"olicher-Nijenhuis bracket is needed. This process partly relies on the Weitzenb\"ock formula \eqref{wz1} for symmetric $2$-tensors.
Integrating the right hand side of the first equation in Lemma \ref{prod-1} we obtain

\begin{pro} \label{fin-braN}
Assume that $h$ and $ H$ belong to $\Gamma (S^2M)$. Then 
\begin{equation*}
\begin{split}
\la \nabla_{e_i}^gh \circ \nabla_{e_i}^gh,H \ra_{L^2}  \, -  \, \la \di_{\nabla^g}h, & H\sharp \di_{\nabla^g}h \ra_{L^2} \;=\;- \, \la \delta^g[h,h]^{\FN},H\ra_{L^2}+2\la 
\delta^{g}[h,H]^{\FN},h\ra_{L^2}\\
&+\la \frac{1}{2}(\widetilde{\Delta}_Eh^2 \, + \, \delta^{\star_g}\di \tr h^2)-
\frac{1}{2}\{h,\widetilde{\Delta}_Eh+\delta^{\star_g}\di \tr h\}-Eh^2,H\ra_{L^2}.
\end{split}
\end{equation*}
\end{pro}
\begin{proof}
First we deal with the second summand in the expression to compute. According to the definition of $[h,H]^{\FN}$ in 
\eqref{bra-nac2} we have 
\begin{equation*}
\begin{split}
g(\di_{\nabla^g}h,  H\sharp \di_{\nabla^g}h) & \, =-g(h \sharp \di_{\nabla^g}H,\di_{\nabla^g}h) \, + \, g(\di_{\nabla^g}\{h,H\}-2[h,H]^{\FN},\di_{\nabla^g}h)\\
& \, = -g(\di_{\nabla^g}H,h \sharp \di_{\nabla^g}h) \, + \, g(\di_{\nabla^g}\{h,H\}-2[h,H]^{\FN},\di_{\nabla^g}h)
\end{split}
\end{equation*}
since the map $\alpha \in \Lambda^2(M,TM) \mapsto h \sharp \alpha \in \Lambda^2(M,TM)$ is symmetric with respect to the inner product induced by $g$. Because $h \sharp \di_{\nabla^g}h=\di_{\nabla^g}h^2-[h,h]^{\FN}$ by \eqref{bra-nac}, after integration over $M$ we find that 
\begin{equation*}
\begin{split}
\la \di_{\nabla^g}h,H\sharp \di_{\nabla^g}h \ra_{L^2}&=\la \delta^g[h,h]^{\FN},H\ra_{L^2}-2\la 
\delta^{g}[h,H]^{\FN},h\ra_{L^2}\\
&
\qquad \qquad \qquad \qquad 
-\la H, \delta^g\di_{\nabla^g}h^2\ra_{L^2}+\la \{h,H\},\delta^g\di_{\nabla^g}h\ra_{L^2}.
\end{split}
\end{equation*}
By also taking into account that the operator $H \in S^2M \mapsto \{h,H\} \in S^2M$ is symmetric we further get 
\begin{equation*}
\la \nabla_{e_i}^gh \circ \nabla_{e_i}^gh,H \ra_{L^2}-\la \di_{\nabla^g}h,H\sharp \di_{\nabla^g}h \ra_{L^2}=-\la \delta^g[h,h]^{\FN},H\ra_{L^2}+2\la 
\delta^{g}[h,H]^{\FN},h\ra_{L^2}+\la S,H \ra_{L^2}
\end{equation*}
where $S=\nabla_{e_i}^gh \circ \nabla_{e_i}^gh+\delta^g\di_{\nabla^g}h^2-\{h,\delta^g\di_{\nabla^g}h\}$. This quantity is determined by using the 
following two observations; firstly we have the operator identity 
$$
\nabla_{e_i}^gh \circ \nabla_{e_i}^gh \; = \; \frac{1}{2}\{h,(\nabla^g)^{\star} \nabla^g h\} \, - \, \frac{1}{2}(\nabla^g)^{\star} \nabla^g h^2.
$$ 
Secondly the Weitzenb\"ock 
formula \eqref{wz1} can be equivalently written as
$$ 
\delta^g \circ \di_{\nabla^g} \, -\, \frac{1}{2}(\nabla^g)^{\star} \nabla^g \; =\; \frac{1}{2}(\widetilde{\Delta}_E +\delta^{\star_g} \circ \di \circ \tr) \, - \, \di \circ \delta^g \, + \, E.
$$
Combining these two facts leads to 
$$
g(S,H) \; = \; g(\frac{1}{2}(\widetilde{\Delta}_Eh^2  +  \delta^{\star_g}\di \tr h^2)  - 
\frac{1}{2}\{h,\widetilde{\Delta}_Eh  +\delta^{\star_g}\di \tr h\}-Eh^2,H)$$ and the claim follows.
\end{proof}
\subsection{The second variation of the Ricci tensor} \label{main}
We assemble the previously developed material into an intrinsic description of the
second variation of the Ricci tensor. 
The primary object entering the latter 
is build from the Fr\"olicher-Nijenhuis bracket $[\cdot, \cdot]^{\FN}$ as follows. Consider the second order, linear differential operator $\mathbf{v}:\Gamma(S^2M) \times \Gamma( S^2M) \to \Gamma(S^2M)$ 
given by 
\begin{equation} \label{def-v}
\la \mathbf{v}(h_1,h_2),h_3\ra_{L^2}
\,=\,
\la \delta^g[h_1,h_2]^{\FN},h_3\ra_{L^2}+\la \delta^g[h_2,h_3]^{FN},h_1\ra_{L^2}+
\la \delta^g[h_3,h_1]^{\FN},h_2\ra_{L^2}.
\end{equation}
Because the Fr\"olicher-Nijenhuis bracket is symmetric, $\bfv$ is symmetric in each pair of entries 
\begin{equation*}
\la \mathbf{v}(h_1,h_2),h_3\ra_{L^2}=\la \bfv(h_2,h_1),h_3\ra_{L^2}=\la \bfv(h_1,h_3),h_2\ra_{L^2}.
\end{equation*}
This algebraic property will play an essential role in determining the gauge invariance properties of the operator 
$\bfv$. 
\begin{teo} \label{int-1}
Let $(M,g)$ be an Einstein manifold with $\ric^g=Eg$ and assume that the curve of metrics $g_t$ satisfies $g_0=g$. 
The $g$-symmetric component in 
$2\left.\frac{\di^2}{\di\!t^2}\right|_{t=0}
\Ric^{g_t}$ is given by 
\begin{equation*}
\begin{split}
\widetilde{\Delta}_E(\ddot{h} - \frac{3}{2}  h^2) + \bfv(  h, h) \, - \,  & E  h^2-\frac{1}{2}\delta^{\star_g}\di\!\tr( h^2)\\
&+ \, 2\delta^{\star_g}(2 h\delta^g h+  h\D^g  h)-\nabla^g_{\grad (\tr  h)}  h-\frac{1}{2}\{  h,\widetilde{\Delta}_E  h+\delta^{\star_g}\di\!\tr(  h)\}
\end{split}
\end{equation*}
where $h=g^{-1}\dot{g}$ and $\ddot{h}=g^{-1}\ddot{g}$.
\end{teo}
\begin{proof}
In the second variation formula for the Ricci tensor in Proposition \ref{prel-fo} we replace the vector field $A$ by its value given in \eqref{comp-A} 
and the last line by \eqref{combo}. Next, we use Proposition \ref{fin-braN}. 
Then a purely algebraic calculation which solely consists in gathering terms shows that the claim holds in an $L^2$- sense. As both 
quantities in the statement are defined at each point of $M$ it follows that the claim also holds pointwise.
\end{proof}

\begin{rema}\label{2var}
Note that by Remark \ref{skew-part} the second derivative of $\Ric^{g_t}$ is completely symmetric if its first derivative vanishes. Hence, in this case 
the theorem above gives an explicit second variation formula for the Ricci tensor $\Ric^{g_t}$.
\end{rema}

As the main application of  Theorem \ref{int-1} we derive below the explicit equation that 
characterises second order Einstein deformations.

\begin{teo} \label{ED2}
Let $(M,g)$ be an Einstein manifold with $\ric^g=Eg$ and let $g_t$ be a curve of metrics starting at $g$, with $h=g^{-1}\dot{g}$ and $\ddot{h}=g^{-1}\ddot{g}$. 
Then $g_t$ is a second order Einstein  deformation if and only if $\widetilde{\Delta}_Eh=0$ and  the following equation holds
\begin{equation*} 
\begin{split}
\widetilde{\Delta}_E(\ddot{h}-\frac{3}{2}h^2) \;  =  \; -\bfv(h,h) \, + \, Eh^2 & \, +  \,  \frac{1}{2}\delta^{\star_g}\di\!\tr(h^2)\\
&- \, 2\delta^{\star_g}(2h\delta^gh \, + \, h\D^g h) \, + \, \nabla^g_{\grad (\tr h)}h \, + \, \frac{1}{2}\{h,\delta^{\star_g}\di\!\tr(h)\} .
\end{split}
\end{equation*}
\end{teo}
\proof
If $g_t$ is a second order Einstein deformation than it is in particular a first order deformation, which by Proposition \ref{first} is equivalent to
$  
\left.\frac{\di}{\di\!t}\right|_{t=0}
\Ric^{g_t} = 0$ and also to $\widetilde{\Delta}_Eh=0$. Then the formula in Remark \ref{skew-part}
shows that the second derivative of $\Ric^{g_t}$ is completely symmetric. Hence, its vanishing is equivalent to the
vanishing of the pointwise expression given in Theorem \ref{int-1} and the claimed equation  follows. Conversely, $\widetilde{\Delta}_Eh=0$
implies that $g_t$ is a first order Einstein deformation. Then Remark \ref{skew-part} and the equation above
together with Theorem \ref{int-1}  implies that
$g_t$ is also a second order Einstein deformation.
\qed

\medskip

If, after a gauge transformation, $h$ is normalised to be trace and divergence free the equation in Theorem \ref{ED2} clearly simplifies further, in the sense that the second displayed line therein vanishes. 
However we avoid this type of symmetry breaking since it is too restrictive for most of the considerations and examples in this paper.

\subsection{The Koiso obstruction}\label{koiso}
As an immediate consequence of Theorem \ref{int-1} we obtain a necessary condition for  the integrability to second order. Indeed, 
let $h \in \scrE(g)$ be integrable to second order and  let $g_t$ be a second order Einstein deformation with $h = g^{-1} \dot g$. Then 
taking the $L^2$ trace in Theorem \ref{int-1} and substituting the definition of $\bfv$ leads to
\be\label{koiso-1}
0= \la 2
\left.\frac{\di^2}{\di\!t^2}\right|_{t=0}
\Ric^{g_t}, h\ra_{L^2}
=
\la  \bfv(h,h) - E h^2, h \ra_{L^2} 
=
\la 3 \delta^g[h,h]^{\FN}  - E h^2, h \ra_{L^2} ,
\ee
where the assumptions on $h$, i.e. that $\widetilde \Delta_E h = 0, \, \delta^g h = 0$ and $\tr_g h = 0$, give the strong simplification
of the variation formula of Theorem \ref{int-1}.

This reformulates in an intrinsic way   Koiso's obstruction as we will see now. The Koiso obstruction is written as $\la \mathrm{E}''_g(h, h), h \ra_{L^2} = \int_MP(h)\vol = 0$, 
where a local calculation (see Koiso \cite[Lemma 4.3]{Ko1}) 
shows that $P(h)$  is defined via
\begin{equation}\label{Ph}
2P(h) \;:= \; 3  g(\nabla^{2}_{e_i,he_i}h,h) \, - \, 6 g ((\nabla^2_{e_i,e_j}h)he_i,he_j) \, + \, 2 E \, \tr(h^3),
\end{equation} 
where   $\{  e_i \}$ is a local orthonormal basis and $\nabla^2= (\nabla^g)^2$. Since $\mathrm{E}''_g(h, h)$ is given by the second variation of the Ricci tensor (see \eqref{E2})
we can rewrite the Koiso obstruction using \eqref{koiso-1}. 
\begin{pro} \label{div-bra}
For any $h \in \scrE(g)$ we have
$$ 
2\int_M P(h)\vol \;  =\;  3\langle \delta^g [h,h]^{\FN},h\rangle_{L^2}  \,  -   \, E\int_M\tr(h^3)\vol.
$$
\end{pro}

For other applications it is perhaps interesting to note how the summands in the local expression of $P(h)$ given
in \eqref{Ph} can be rewritten in an index free notation. This we will do in Appendix \ref{koiso-plus} where we also give
a direct proof of Proposition \ref{div-bra}.

%
\section{Deformation theory to second order}
%

In this section we will derive the necessary and sufficient condition of Theorem \ref{main1} for an infinitesimal Einstein deformation $h \in \scrE(g)$ to be integrable to
second order.

\subsection{Second order integrability} \label{pfm}
To solve the Einstein equation to $2$nd order  preliminary information regarding the Hodge theory of the Laplacian $\widetilde{\Delta}_E$ is needed. This essentially relies on the properties of the well-known $L^2$-orthogonal splitting 
\begin{equation} \label{ssplit}
\Gamma(S^2M) \, = \, \TT(g) \, \oplus \, \{\delta^{\star_g}X+f\id : X \in \Gamma(TM), \, f \in  C^{\infty}M \}.
\end{equation} 

\begin{pro} \label{Hodge1}
The following hold
\begin{itemize}
\item[(i)] $\D^g \circ \widetilde{\Delta}_E=0$
\item[(ii)] $\ker \D^g \,=\, \TT(g) \, \oplus \,\widetilde{\Delta}_E\{f\id : \int_Mf\vol=0\}  \, \oplus \, \bbR\id$, where the sum is $L^2$-orthogonal 
\item[(iii)] For any $v \in \Gamma(S^2M)$ there exists $u \in \Gamma(S^2M)$ solving $\widetilde{\Delta}_Eu=v$ 
if and only if 
\begin{equation*}
\begin{split}
&v \perp_{L^2} \mathscr{E}(g) \ \mathrm{and} \ \D^g v=0 \ \mathrm{when} \ E \neq 0 \\
&\mathrm{respectively}\\\
&v \perp_{L^2} \scrE(g) \ \mathrm{and} \ \D^g v=0, \ \tr_{L^2}v=0 \ \mathrm{when} \ E=0. 
\end{split}
\end{equation*}
\end{itemize}
\end{pro}
\begin{proof}
First record that direct computation shows for all $X \in \Gamma(TM)$ and $f \in C^{\infty}M$ that 
\begin{equation}\label{eq1}
\D^g(\delta^{\star_g}X+f\id)  = (\Delta^g-2E)X+(n-2)\di\!f 
\end{equation}
\vspace{-4ex}
\begin{equation}\label{eq2}
 \widetilde{\Delta}_E(f\id) = (\Delta^gf-2Ef)\id-(n-2)\delta^{\star_g}\di\!f. 
\end{equation}
The first equation follows from $\delta^g (f \id) = -\di\!f$, the Weitzenb\"ock formula \eqref{wz2}
and the relation $\tr (\delta^{\star_g} X) = - \di\!^{\star}X$. To prove the second equation we first recall that $ \widetilde{\Delta}_E = \Delta_E - \delta^{\star_g}\circ  \D^g$ and $\D^g = 2 \delta^g + \di \circ \tr $ according to the definitions in section \ref{dwz}; then we use \eqref{eq1} with $X=0$ and the relation $\Delta_E (f \id) = (\Delta^g f - 2 E f )\id$. 

The splitting \eqref{ssplit} is preserved by the Laplacian $\widetilde{\Delta}_E$ which additionally satisfies 
\begin{equation}\label{vanish}
\widetilde{\Delta}_E=\Delta_E \ \mathrm{on} \ \TT(g) 
\qquad  \mathrm{and} \qquad  \widetilde{\Delta}_E \circ \delta^{\star_g}=0.
\end{equation}
The last identity is contained in \eqref{ga-ba}. Then the  claim in (i) follows  by direct computation, since 
$\D^g$ vanishes on $\TT(g)$ and 
$
\D^g ( \widetilde{\Delta}_E (\delta^{\star_g}X+f\id)) = \D^g ( \widetilde{\Delta}_E (f \id)) = 0
$
using \eqref{eq2}.

\noindent
(ii) Since the Bianchi operator $\D^g$ vanishes on $\TT(g)$ it is enough to solve $\D^g(\delta^{\star_g}X+f\id)=0$, or equivalently 
$(\Delta^g-2E)X+(n-2)\di\!f=0$  by \eqref{eq1}. Writing $X = \di\!F + K$ with $\di^{\star}\!K=0$, using the Hodge decomposition of $\Omega^1(M)$, 
we obtain the two equivalent equations 
$
(\Delta^g - 2E)K = 0
$
and
$
(\Delta^g - 2E) \di\!F + (n-2)\di\!f = 0
$.
Since we can choose $F$ with  $\int_M F \vol_g = 0$ it follows that $(\Delta^g - 2E) F + (n-2) f =  (n-2)\int_M f \vol_g$.
Moreover, by  Lichnerowicz's theorem, the first equation implies that $K$ is a Killing vector field and in particular that $\delta^{\star_g}K=0$.
Hence, we can write $\delta^{\star_g} X + f \id = \delta^{\star_g} \di\! F + f \id$. Next we replace $ \delta^{\star_g} \d F $ using \eqref{eq2}, applied to $F\id$, and
once again with the second equation from above. Finally we obtain $\delta^{\star_g}X+f\id=(\int_Mf\vol_g)\id-\frac{1}{n-2}\widetilde{\Delta}_E(F\id)$ 
and the claim is proved.

\noindent
(iii) 
Splitting $u=u_0+u_1$ respectively $v=v_0+v_1$ according to \eqref{ssplit}, the equation under scrutiny decouples as  
$\Delta_Eu_0=v_0$ and $\widetilde{\Delta}_Eu_1=v_1$. 
The integrability condition for the first is $v_0 \perp_{L^2} \scrE(g)$, since $\Delta_E$ is elliptic and self-adjoint. 
The conditions $\D^g v_0 = 0 = \tr \, v_0$ are satisfied by the definition of  $\scrE(g)$. The second equation implies
$\D^g v_1 = 0$ by part (i) and $v_1 \perp_{L^2} \scrE(g)$ by the definition of $v_1$. 

We still have to derive the additional condition in the case $E=0$.
According to \eqref{ssplit} we
can write $u_1 = \delta^{\star_g} X + f \id$ and thus by \eqref{eq2} and the second equation in \eqref{vanish} we have 
$\widetilde \Delta_E u_1 = (\Delta^g f) \id - (n-2) \delta^{\star_g} \di\! f$ assuming $E=0$. Taking the trace and applying
$\tr (\delta^{\star_g} X ) = - \di^{\star}\!X$ we get 
$\tr_{L^2} v =\tr_{L^2} v_1=\tr_{L^2}( \widetilde{\Delta}_Eu_1)=(2n-2)\int_M (\Delta^g f) \vol_g= 0$. 

Conversely, since  $\D^g v_1 = 0$ we can use the description of $\ker  \D^g $ given in part (ii). In the case 
$E \neq 0$ it is enough to
verify that the summand $\RM \id$ is in the image of $\widetilde \Delta_E$; this follows from
\eqref{eq2} as we can write $\id = - \frac{1}{2E} \widetilde \Delta_E \id$. In the case $E=0$ we have the additional
condition $\tr_{L^2} v_1 = 0$ which implies that $v_1$ does not have a component on the summand $\RM \id$.
\end{proof}

\begin{rema}
Part (i) in the Proposition \ref{Hodge1} also follows directly by differentiating the contracted Bianchi identity $\D^{g_t} \Ric^{g_t}=0$ on any curve of metrics $g_t$ starting at $g$. 
The full force of this argument, which permits avoiding lengthy computations, will be exploited in the proof below. Note however that \eqref{eq1} and \eqref{eq2} are needed to establish (ii) and (iii) in Proposition 
\ref{Hodge1}.
\end{rema}
The proof of the main result in this section can now be completed.
\begin{teo} \label{int-33}
Let $(M^n,g), n \geq 3$ be  an Einstein manifold with $\ric^g=Eg$. An element $h \in \scrE(g)$ is integrable to second order if and only if 
\begin{equation} \label{IF1}
- \bfv(h,h)  +  Eh^2   \;\perp_{L^2} \;   \scrE(g).
\end{equation}
Equivalently, 
\begin{equation} \label{IF2}
\la \delta^g[h,h]^{\FN},H \ra \, + \, 2\la \delta^g[h,H]^{\FN} ,h\ra_{L^2} \; = \; E \, \tr_{L^2}(h^2\circ H) \quad \mathrm{for \ all} \ H \in \scrE(g).
\end{equation}
\end{teo}
\begin{proof}
An element $h \in \scrE(g)$ is integrable to second order if and only if there exists a curve of metrics
$g_t$ with Taylor expansion $g^{-1}g_t=\id+th+\frac{t^2}{2}\ddot h+o(t^3)$ which is a second order Einstein deformation. By Theorem \ref{ED2} this is equivalent to having the pair of symmetric tensors 
$$u:=\ddot{h}-\frac{3}{2}h^2 \ \mathrm{and} \ v:=-\bfv(h,h)+Eh^2+\frac{1}{2}\delta^{\star_g}\di\!\tr(h^2)$$
satisfy the equation 
$\widetilde{\Delta}_Eu=v.$
Note that the other summands
in the equation of Theorem \ref{ED2} vanish since $h  \in \scrE(g) $, in particular it is assumed to be a TT-tensor. 

Thus in order to prove the statement it suffices to investigate when $v$, depending on $h$ as indicated above, belongs to $\mathrm{Im} \widetilde{\Delta}_E$. To this extent, we now check under which requirements on a given $h \in \scrE(g)$, the two sets of integrability conditions, as given in part (iii) of Proposition \ref{Hodge1}, are indeed satisfied for $v$. 

\noindent
\textit{Condition $\D^gv=0$ is satisfied for all $h \in \scrE(g)$.}\\
In order to verify the condition we consider for small $t$  the curve of metrics $G_t$ given 
by $g^{-1}G_t=\id+th$  and differentiate the identity $\D^{G_t}\Ric^{G_t}=0$ to second order, at $t=0$. 
Since $h \in \scrE(g) \subset \Ker \widetilde{\Delta}_E$ we have $
\left.\frac{\di}{\di\!t}\right|_{t=0}
\Ric^{G_t}=0$ and it follows that 
$$\D^g(
\left.\frac{\di^2}{\di\!t^2}\right|_{t=0}
\Ric^{G_t})+E
\left (\left.\frac{\di^2}{\di\!t^2}\right|_{t=0}
\D^{G_t} \right )\id=0.$$ Because $\D^{G_t}\id=0$,  differentiation shows that 
$
\left (
\left.\frac{\di^2}{\di\!t^2}\right|_{t=0}
\D^{G_t} \right )\id=0$. Hence by the second variation formula in Theorem 
\ref{int-1}  
we get  
$$\frac{3}{2}\widetilde{\Delta}_E h^2   -  \bfv(h,h)  +  Eh^2  +  \frac{1}{2}\delta^{\star_g}\di\!\tr(h^2)=\frac{3}{2}\widetilde{\Delta}_E h^2+v \in \ker \D^g.$$
As $\mathrm{Im}\widetilde{\Delta}_E \subseteq \ker \D^g$ by part (i) in Proposition \ref{Hodge1}, it follows that 
$v \in \ker \D^g$.

\noindent
\textit{Condition $\tr_{L^2}v=0$, in case $E=0$, is satisfied for all $h \in \scrE(g)$.}\\
Indeed, when $E=0$, taking into account 
$
\tr\, (  \delta^{\star_g} \di  \tr(h^2)) = - \Delta^g\tr(h^2)
$ 
yields the vanishing of
$$ 
\tr_{L^2}v=-\tr_{L^2}\bfv(h,h)=-\la \delta^g[h,h]^{\FN},\id \ra_{L^2}-2\la \delta^g[h,\id]^{\FN},h\ra_{L^2}
$$
since $[h,\id]^{\FN}=0$ \and $\di_{\nabla^g}\id=0$. 

\noindent
\textit{Condition $v \perp_{L^2} \scrE(g)$ is equivalent to \eqref{IF1} whenever  $h \in \scrE(g)$.}\\
This is the remaining integrability requirement in Proposition \ref{Hodge1},(iii). This condition is equivalent to \eqref{IF1}
since $\delta^g$ vanishes on $\scrE(g)$. 

We have thus proved that whenever $h \in \scrE(g)$ satisfies 
\eqref{IF1} then all the conditions of  Proposition \ref{Hodge1}, part (iii) are satisfied
for the symmetric tensor $v$ defined above. Hence the metric $g_t$ given by 
$g^{-1}g_t=\id+th+\frac{t^2}{2}(u+\frac{3}{2}h^2)$ is a second 
order Einstein deformation starting at $h$; here $u$ is the solution of the equation $\widetilde{\Delta}_Eu=v$ 
provided, as explained above, by Proposition \ref{Hodge1},(iii). 

Conversely, if $h \in \scrE(g)$ is integrable to second order, then due to the equation 
$\widetilde{\Delta}_E u=v$ we get $v \perp_{L^2} \scrE(g)$, again 
by Proposition \ref{Hodge1}, (iii). As just seen above this amounts to having \eqref{IF1} satisfied for $h$.

Finally, the equivalence of \eqref{IF1} and \eqref{IF2} follows by expanding $\bfv$ according to \eqref{def-v}. 
\end{proof}

A second, slightly more explicit description of the integrability to second order condition in $\mathscr{E}(g)$ is displayed below. 
\begin{teo} \label{int-23}
An infinitesimal Einstein deformation $h \in \mathscr{E}(g)$ is integrable to second order if and only if it satisfies 
\begin{equation} \label{int-fin}
2\la \di_{\nabla^g} h, h \sharp \di_{\nabla^g}H\ra_{L^2} \,  +  \,  \la \di_{\nabla^g}h, H \sharp \di_{\nabla^g} h \ra_{L^2} 
\; = \; \langle \{h,\ring{R}h\} \, + \, (\ring{R}+2E)h^2,H\rangle_{L^2} 
\end{equation}
whenever $H \in \scrE(g)$.
\end{teo}
\begin{proof}Letting $h$ and $H$ belong to $\scrE(g) $ we expand $\bfv$ according to \eqref{def-v}, then express the Fr\"olicher-Nijenhuis bracket by means of 
\eqref{bra-nac2}. Using that  $\alpha \in \Lambda^2(M,TM)\mapsto h \sharp \alpha \in \Lambda^2(M,TM)$ 
and $H \in S^2M \mapsto \{h,H\} \in S^2M$  are symmetric maps leads to 
\begin{equation*}
\begin{split}
&\la \bfv(h,h),H\ra_{L^2}  \;   = \;    \, \la [h, h]^{\FN},   \di_{\nabla^g} H\ra_{L^2}  \, + \,  2 \, \la  [h, H]^{\FN},  \di_{\nabla^g} h\ra_{L^2}  \\[1ex]
 & = - 2  \la h \sharp \di_{\nabla^g} h, \di_{\nabla^g}H\ra_{L^2}  \,+\,  \la \di_{\nabla^g}h^2, \di_{\nabla^g}H\ra_{L^2}    
\,-\,   \la \di_{\nabla^g}h, H \sharp \di_{\nabla^g} h \ra_{L^2}
 \,+\, \la \di_{\nabla^g}\{h,H\}, \di_{\nabla^g}h\ra_{L^2}\\[1ex]
 &=
  - 2  \la \di_{\nabla^g} h, h \sharp  \di_{\nabla^g}H\ra_{L^2} \, -  \,  \la \di_{\nabla^g}h, H \sharp \di_{\nabla^g} h \ra_{L^2}
\,+\,
\la h^2, \delta^g \di_{\nabla^g}H\ra_{L^2} \, + \,\la \{h,\delta^g \di_{\nabla^g}h\},H\ra_{L^2}.
\end{split}
\end{equation*}
Because $\delta^g \circ \di_{\nabla^g}=\ring{R}+E$ on $\scrE(g)$ by \eqref{wz1} the last two summands satisfy
\begin{equation*}
\la h^2, \delta^g \di_{\nabla^g}H\ra_{L^2}+\la \{h,\delta^g \di_{\nabla^g}h\},H\ra_{L^2}\\
= \la \ring{R}h^2+\{h, \ring{R}h\}+3Eh^2,H\ra_{L^2}.
\end{equation*}
As Theorem \ref{int-33} ensures that $h$ is integrable to second order if and only if it satisfies the equation
$\la \bfv(h,h),H\ra_{L^2}=E\la h^2,H \ra_{L^2}$ for all $H \in \scrE(g)$, the claim follows.
\end{proof}
In some situations this reformulation has the advantage to be amenable to further interpretation based on 
the algebraic properties of the action $h \sharp \cdot : \Lambda^2(M,TM) \to \Lambda^2(M,TM)$. Examples in that direction will be given in the next section.

%
\section{Einstein deformations on  K\"ahler manifolds}
%
Let $(M^{2m},g,J)$ be a compact  K\"ahler-Einstein manifold with  $\ric^g=Eg$ and K\"ahler form  
$\omega:= g(J\cdot,\cdot)$.  Then the  bundle of symmetric $2$-tensors splits as $S^2M=S^{2,+}M \oplus S^{2,-}M$, where 
$S^{2,\pm}M=\{h \in S^2M : hJ=\pm Jh\}$.  Using the  notation 
$
\mathscr{E}^{\pm}(g) = \mathscr{E}(g) \cap \Gamma(S^{2,\pm}M)
$
the space of infinitesimal Einstein deformations of $g$ splits as 
\begin{equation} \label{split-E}
\mathscr{E}(g)=\mathscr{E}^{+}(g) \oplus \mathscr{E}^{-}(g)
\end{equation}
As showed by  Koiso in  \cite[Proposition 7.3]{Ko2} the spaces $\mathscr{E}^{\pm}(g)$ are given by
$$
\mathscr{E}^{+}(g)=\{F\circ J:F \in \Omega^{1,1}_0M \cap \ker(\Delta^g-2E) \cap \ker \di^{\star_g}\}$$ 
and 
$$
\mathscr{E}^{-}(g)=\{ h \in \Gamma(S^{2,-}M) : \delta^gh=0 \ \mathrm{and} \ \di_{\nabla^g}h=J\di_{\nabla^g}h \},
$$
where  $J\di_{\nabla}h=\di_{\nabla}h(J \cdot, J\cdot)$.
An outline of the main ingredients entering the proof of \eqref{split-E} is as follows. The description of  
$
\scrE^{-}(g) =  \mathscr{E}(g) \cap \Gamma(S^{2,-}M) =
\ker \Delta_E \cap \Gamma(S^{2,-}M)
$ is due to the identity 
$$
 \la \Delta_Eh,h\ra_{L^2} \,=\, \frac{1}{2}\Vert \di_{\nabla^g}h  - J\di_{\nabla^g}h\Vert_{L^2}^2 \, + \, 2\Vert \delta^gh \Vert^2_{L^2}
$$
with $h \in \Gamma(S^{2,-}M)$. For $\mathscr{E}^{+}(g)$ we have a bundle isomorphism 
$h \in S^{2,+}M  \mapsto hJ \in \Lambda^{1,1}M$.  It is parallel and  thus commutes with the Einstein operator in the sense that the identification 
$h \mapsto h J$ maps $\Delta_E$ to $\Delta^g - 2E$.
As the Einstein operator preserves the type decomposition of symmetric tensors it follows that
$\ker \Delta_E \cong \left ( \ker(\Delta^g-2E) \cap \Omega^{1,1}M \right )\oplus \scrE^{-}(g)$.

\begin{rema} \label{int-delbar}
After complexification the space $\Gamma(S^{2,-}M)$ corresponds to a subspace of
$\Omega^{0,1}(M,T^{1,0}M)$. The operator 
$\di_{\nabla}-J\di_{\nabla}$ corresponds to 
$\overline{\partial}:\Omega^{0,1}(M,T^{1,0}M) \to 
\Omega^{0,2}(M,T^{1,0}M).$ Therefore the Weitzenb\"ock formula above 
 shows that the space $\scrE^{-}(g)$ is contained in the space of infinitesimal complex
deformations of $J$. Again, this is done up to terms of the form $\L_XJ$ with 
$X \in TM$, which automatically belong to the kernel of the operator $\di_{\nabla^g}-J\di_{\nabla^g}$. To see this, we use formula \eqref{bra-nac2} to check that 
$$
 2[h,J]^{\FN}(J \cdot, \cdot) \, = \, \di_{\nabla^g}h \, - \, J\di_{\nabla^g}h
$$
whenever $h \in S^{2,-}M$. Since requiring $J$ to be integrable amounts to 
$[J,J]^{\FN}=0$, after linearisation it follows that the space of infinitesimal complex deformations of $J$ contains $\scrE^{-}(g)$.  
\end{rema}

\begin{rema}\label{rmk3.1}
When $m=2$ and $M$ is positively oriented by $\omega^2$ we have $\Lambda^{1,1}_0M=\Lambda^{-}M$. As 
$\ker \d=\ker \di^{\star_g}$ on $\Omega^{-}M$ it follows that $\scrE^{+}(g)=0$. 
In higher dimensions it is unknown if there are large classes of inhomogeneous K\"ahler-Einstein manifolds  with non-vanishing $\scrE^{+}(g)$. 

Due to Koiso (see \cite[Theorem 1.1]{Ko3}, see also \cite[Proposition 2.40]{GaGo}),  the only symmetric 
K\"ahler-Einstein manifolds  admitting infinitesimal Einstein deformations are the complex Grassmannians. Here we have 
$\scrE(g) = \scrE^{+}(g) $ and the space $\scrE(g)$ is isomorphic to the space of Killing vector fields, i.e. it is in particular non-empty. We will discuss it in more detail in the next section for the complex $2$-plane Grassmannian.
\end{rema}

\subsection{Negative scalar curvature} \label{sec-neg}
If  $E<0$ we have $\mathscr{E}^{+}(g)=0$ hence $\mathscr{E}(g)=\mathscr{E}^{-}(g)$, since the Hodge-Laplace operator is non-negative
and cannot have the eigenvalue $2E$.

Results on the integrability, through a curve of Einstein metrics, of elements in $\scrE(g)$ are usually phrased in terms of properties of the complex structure; see \cite{Dai}.
One assumption that is frequently used is that all infinitesimal complex deformations are integrable. Below we show that 
for integrability to order two no assumption on the complex structure is needed.

\begin{teo} \label{Eneg}
Let $(M^{2m},g,J)$ be a compact  K\"ahler-Einstein manifold with $\ric^g=Eg$ where $E<0$. Then any infinitesimal 
Einstein deformation in $\scrE(g)$ is integrable
to second order.
\end{teo}
\begin{proof}
Throughout the proof we use the notation $\lambda^2:=\{\alpha \in \Lambda^2M : J\alpha=-\alpha \}$,
where here and in the computations below we let $J$ act on forms $\beta \in \Omega^{\star}M$ according 
to the convention $J\beta:=\beta(J \cdot, \ldots, J\cdot)$. Note that $\lambda^2$ is the real part of the
bundle $\Lambda^{(2,0)} M \oplus \Lambda^{(0,2)} M$.

Because the operator $\bfv$ is symmetric in each pair of variables it suffices to show that \eqref{int-fin} 
holds, with $H=h$, whenever $h \in \scrE(g)$. In fact we will show that both sides of   \eqref{int-fin}  vanish. This is done by entirely algebraic arguments based on the 
equality $\scrE(g)=\scrE^{-}(g)$ as follows. Since $hJ=-Jh$ the square $h^2$ commutes with $J$ showing that 
$
\la h^2, h \ra_{L^2} = \tr_{L^2}h^3=0
$. Because $(g,J)$ is K\"ahler, the curvature action $\ring{R}$  preserves tensor type, 
$$ 
\ring{R} \left ( S^{2,\pm}M \right )\subseteq S^{2,\pm}M.
$$
Thus elements of the form $\{\ring{R}h,h\}$ and $\ring{R}h^2$ belong to $S^{2,+}M$ and hence must satisfy
the orthogonality relation
$\la \{\ring{R}h,h\}+\ring{R}h^2, h \ra_{L^2}=0$. Finally, having $h \in \scrE^{-}(g)$ means by the definition of $\scrE^{-}(g)$ that 
$\di_{\nabla^g}h \in \Gamma( \Lambda^{1,1} \otimes T)$ and thus $h \sharp \di_{\nabla^g}h \in  \Gamma(\lambda^{2} \otimes T)$ 
since $hJ+Jh=0$. This observation implies that $\la h \sharp \di_{\nabla^g}h,\di_{\nabla^g}h\ra_{L^2}=0$ and the proof is complete.
\end{proof}

\subsection{Positive scalar curvature} \label{pos}
We will show that contrary to Theorem \ref{Eneg}, deformations in $\scrE^{+}(g)$ are generally obstructed when 
$E>0$. Proving this relies on the following general computation for the Fr\"olicher-Nijenhuis bracket of
Hermitian symmetric tensors, providing a formula for the operator  $\bf v$ on symmetric tensors of the form $h = F \circ J \in  \scrE^{+}(g)$.
Note that for the next lemma we do not assume $M$ to be compact.
\begin{lema} \label{FN-J}
Let $(M^{2m},g,J)$ be a K\"ahler manifold  and let $F \in \Omega^{1,1}_0M$ satisfy $\di^{\star}F=0$. Then
$$ 
g([F\circ J,F\circ J]^{\FN},\di_{\nabla^g}(F\circ J)) \, = \, -2g(\omega \wedge \di\!F, F \wedge \di\!F) \, + \, \frac{3}{2}
g(\di\!F, \di\! L^{\star}_{\omega}(F \wedge F)),
$$
where $ L^{\star}_{\omega}$ denotes the contraction with the K\"ahler form $\omega$.
\end{lema}
\begin{proof}
Using \eqref{bra-na} we write the Fr\"olicher-Nijenhuis bracket for symmetric endomorphisms of type $F \circ J $ as
\begin{equation*}
\begin{split}
g([F\circ J & ,F\circ J]^{\FN}(X,Y),Z)\\
= & \quad g(-(\nabla_{FJX}^gF)JY\,  + \,  (\nabla_{FJY}^gF)JX,Z)  \,  + \,  g((\nabla^g_XF)Y-(\nabla^g_YF)X,FZ)\\
=&-\alpha(JX,JY,Z) \, + \,  g((\nabla_{FZ}^gF)X,Y) \, + \, g((\nabla^g_XF)Y-(\nabla^g_YF)X,FZ)\\
=&-\alpha(JX,JY,Z)+\di\!F(X,Y,FZ) ,
\end{split}
\end{equation*}
where the $3$-form $\alpha$ is defined via $\alpha(X,Y,Z):=\mathfrak{S}_{X,Y,Z}g((\nabla^g_{FX}F)Y,Z)$, with 
$\mathfrak{S}_{X,Y,Z}$ denoting the cyclic sum over $X, Y, Z$. Note that $\alpha$ can also be written as
$
\alpha = e^k \wedge \nabla_{Fe_k}^g F
$.
At the same time we have
\begin{equation*}
\begin{split}
g(\di_{\nabla^g}(F \circ J)(X,Y),Z)=&g((\nabla^g_{JZ}F)X,Y)-\di\!F(X,Y,JZ).
\end{split}
\end{equation*}
Taking the scalar product of these two expressions leads easily to 
\begin{equation*}
\begin{split}
g([F\circ J , F\circ J]^{\FN},\di_{\nabla^g}(F\circ J))  & = g(J \alpha , e^k \wedge \nabla^g_{e_k}F) \, - \,   g(J \alpha, e^k \wedge e_k \,  \lrcorner \, \di\!F)
\\
&
\qquad -\, g(\di\!F, e^k \wedge \nabla_{FJe_k}^gF)  \,  +  \,  g(\di\!F, e^k \wedge FJe_k \, \lrcorner \, \di\!F).
\end{split}
\end{equation*}
This we rewrite using the equations  $\di\!F = e^k \wedge \nabla^g_{e_k}F$ and $ J \alpha =  e^k \wedge \nabla_{FJe_k}^gF$, together with
$e^k \wedge e_k \,  \lrcorner \, \di\!F = 3 \di\!F$. In addition we need the relation   
$L^{\star}_{\omega}(F \wedge \di\!F) = e^k \wedge FJe_k \, \lrcorner  \, \di\!F$, where the contraction 
with the K\"ahler form $\o$ is  
given by  $2L^{\star}_{\omega}=Je_k \lrcorner e_k \lrcorner$. Thus, using that $F$ is primitive, i.e. 
$L^{\star}_{\omega} F = 0$ and $J$-invariant, i.e. $JF = F$ we find
$$
2 L^{\star}_{\omega}(F \wedge \di\!F) \; = \;  2 e^k \wedge FJe_k  \, \lrcorner \, \di\!F \, + \,2 F \wedge [ L^{\star}_{\omega}, \di   ]  F 
\; = \;  2 e^k \wedge FJe_k \, \lrcorner \, \di\!F ,
$$
where we also used the K\"ahler identity $[ L^{\star}_{\omega}, \di   ]  = J \!  \di^{\star}\! J$ and the assumption that $F$ is coclosed.
Combining these facts yields 
\begin{equation*}
\begin{split}
g([F\circ J,F\circ J]^{\FN},\di_{\nabla^g}(F\circ J)) \; = \; - 3 g(\di\!F,J\alpha) \, + \,g(\di\!F,L^{\star}_{\omega}(F \wedge \di\!F)).
\end{split}
\end{equation*}
Furthermore, $2\alpha=\di^{\star}(F \wedge F)$, as it follows from  a short computation based on $-\di^{\star}=e_i \lrcorner \nabla^g_{e_i}$.
It follows that  $2J\alpha=J\di^{\star}\!J(F \wedge F)=L^{\star}_\omega \di(F \wedge F)-\di\!L^{\star}_\omega (F \wedge F)$ by again using 
the K\"ahler identities. Substituting this expression for $J\alpha$ into the last displayed equation the claim follows
\end{proof}

The following algebraic identity will be used frequently in this section and the next. Its proof follows from a  short algebraic computation
based on the definition of $L^{\star}_{\omega}$.
\begin{lema} \label{Ls}
Assume that $F \in \Lambda^2_0 M$. Then $L^{\star}_{\omega}(F \wedge F)=2FJF$.
\end{lema}

On compact K\"ahler-Einstein manifolds of positive scalar curvature the  following theorem gives an integrability criterion to second order for infinitesimal Einstein deformations in $ \scrE^{+}(g)$. In the next section we will use this criterion for studying 
the integrability question for the complex $2$-plane Grassmannian. We have

\begin{teo} \label{E-geq0}
Let  $(M^{2m},g,J)$ be a compact  K\"ahler-Einstein manifold with $\ric^g=Eg$, where $E>0$. If 
$\scrE^{-}(g)=0$ an element $F J \in \scrE^{+}(g)$ is integrable to second order if and only if
\begin{equation*}
\la \omega \wedge \di\!G, F \wedge \di\!F\ra_{L^2} \, + \,  \la \omega \wedge \di\!F, F \wedge \di\!G+G \wedge \di\!F \rangle_{L^2}
\; = \; 8E\la F^2J,G\ra_{L^2}
\end{equation*} 
for all symmetric tensors  $G J \in \scrE^{+}(g)$.
\end{teo}
\begin{proof}
In order to apply Theorem \ref{int-33} we need a formula for 
$
\bfv(h, h, h) = \la \bfv (h, h), h \ra_{L^2}
$
on symmetric tensors $h = FJ \in  \scrE^{+}(g)$. Integrating by parts in Lemma \ref{FN-J} yields 
$$
\frac{1}{3}\bfv(FJ,FJ,FJ)
\; = \; \langle [F J,F J]^{\FN},\di_{\nabla^g}(F J)\rangle_{L^2}
\; = \; - 2 \, \langle \omega \wedge \di\!F, 
F \wedge \di\!F \rangle_{L^2} \, + \,  6 E \, \langle F, F^2 J\rangle_{L^2} .
$$
Here we have used  Lemma \ref{Ls} and the fact that $\Delta F = \di^{\star} \di F = 2 E F$, which is due to 
having $F J \in \scrE^{+}(g)$.
After polarisation, it thus follows that
\begin{equation*}
\begin{split}
&\bfv(F J,F J, G J)
\; =  \frac{1}{3}\left.\frac{\di}{\di\!t}\right|_{t=0}
\bfv( FJ + tGJ, FJ + tGJ, FJ + tGJ)\\
=& -2 \Big ( \la \omega_J \wedge \di\!G, F \wedge \di\!F\ra_{L^2} + 
\la \omega_J \wedge \di\!F, F \wedge \di\!G+G  \wedge \di\!F  \ra_{L^2} \Big )+\; 18E \, \la F^2J,G\ra_{L^2}.
\end{split}
\end{equation*} 
According to \eqref{IF1}  in Theorem \ref{int-33}, and since $\mathscr{E}^{-}(g)$, the element $FJ$ is integrable 
to second order if and only $\bfv(FJ,FJ,GJ)=E\tr_{L^2}(FJ)^2GJ$. However, expressing the scalar product on forms in terms of the trace 
yields 
$\tr (FJ)^2GJ=-\tr(F^2J)G=2\la F^2J, G \ra$ and the claim follows.
\end{proof}

In dimension $6$, which is the first dimension when the space $\scrE^{+}(g)$ does not vanish a priori, this simplifies considerably for algebraic reasons.
\begin{teo} \label{E-geq06}
Assume that $(M^{6},g,J)$ is a compact  K\"ahler-Einstein  manifold with $\ric^g=Eg$ where $E>0$. If 
$\scrE^{-}(g)=0$ an infinitesimal Einstein deformation $FJ \in \scrE^{+}(g)$ is integrable  to second order if and only if
\begin{equation*}
\la F^2J,G\ra_{L^2}=0 \ \mathrm{whenever} \ GJ \in \scrE^{+}(g).
\end{equation*} 
\end{teo}
\begin{proof}
Since for all $ FJ  \in \scrE^{+}(g)$ the $2$-form $F$ is $J$ invariant,  coclosed and primitive we have $L_{\omega}^{\star}\di\!F=0$  
by the K\"ahler identities, i.e.
$ \di F \in \Lambda^{2,1}_0 M \oplus \Lambda^{1, 2}_0 M$. 
Hence, $\omega \wedge \di F = 0$ since the complex dimension is $3$ and $\omega \in \Lambda^{1, 1}M $.
Thus the left-hand side in Theorem \ref{E-geq0} vanishes whence the claim.
\end{proof}

%
\section{The complex $2$-plane Grassmannian} \label{gra}
%
Consider the positive K\"ahler-Einstein manifold $(M^{2m}=\mathrm{Gr}_2(\bbC^{n+2})=\frac{\SU(n+2)}{S(\U(n)\times \U(2))},g,J)$ where $(g,J)$ is the Hermitian symmetric space structure on $M$ with K\"ahler form $\omega_J$. Hence $m=2n$ and the metric $g$ has a quaternion-K\"ahler structure defined by a $\nabla^g$-parallel rank $3$ subbundle $Q\subseteq \Lambda^2M$. 
This follows from  the splitting $\su(n+2 ) = \mathfrak{k}  \oplus  \mathfrak{m}$ with the isotropy algebra 
$\mathfrak{k} = s(\u(2) \oplus \u(n))$ and its orthogonal complement $\mathfrak{m}$  with respect to the Killing form of $\su(n+2)$. 
As $\mathfrak{k}$ embeds into $\Lambda^{1,1}\mathfrak{m}$ via the Lie bracket we obtain Lie algebra subbundles $Q$ and $\mathbb{E}$ of $\Lambda^{1,1}_0M$ isomorphic at each point of $M$ to $\su(2)=\sp(1)$ respectively $\su(n)$ and such that $[Q,\mathbb{E}]=0$ in $\so(TM)$. In addition the bundles $Q$ and $\mathbb{E}$ are holonomy invariant by construction, hence parallel with respect to the Levi-Civita connection $\nabla^g$.
We will use the quaternion-K\"ahler structure $Q$ to describe a geometric way of parametrising the space $\scrE(g)$. 

Recall that a Killing vector field $X $ on a K\"ahler manifold $(M, g)$ is Hamiltonian with respect to the K\"ahler form, i.e.  $\L_X \omega_J=0$.  
When $M$ is simply connected its moment map $z_X$ (also called Killing potential) is uniquely determined from $X \, \lrcorner \, \omega_J=\di\!z_X$ and $\int_M z_X\vol=0$. Note that 
if $g$ is an Einstein metric with Einstein constant $E$, Killing vector fields and their Killing potentials are both in the kernel of $\Delta^g - 2E$ acting on vector fields respectively functions. In particular, we have $\la X, Y\ra_{L^2} = 2E \la z_X, z_Y \ra_{L^2}$ and we see for $E \neq 0$ that if a Killing potential is $L^2$-orthogonal to all
other Killing potentials the corresponding Killing vector field, being orthogonal to all other Killing vector fields, has to vanish.
Below we will make substantial use of this argument.

In particular, $\di\!X^{\flat}$ belongs to $\Omega^{1,1}M$. Denoting with $\alpha_X$ the component of 
$\di\!X^{\flat}$ on 
$\Omega^{1,1}_0M$ we have $\alpha_X = (\di\!X^{\flat})_0 =\di\!X^{\flat}-\frac{2E}{m}z_X\omega_J$. 
Indeed, 
$$g(\di\! X^{\flat}, \omega) = \di^{\star} JX^{\flat} \ \mathrm{and} \ \di (\di^{\star} J X^{\flat} - 2E z_X) = \Delta^g JX^{\flat} - 2E JX^{\flat} = 0,$$ since 
$\Delta^g X^{\flat} = 2 E X^{\flat}$ and $[\Delta^g, J] = 0$ on $1$-forms.
Furthermore 
\begin{equation} \label{dA}
\begin{split}
&\di^{\star}\alpha_X \, = \, \frac{2E(m-1)}{m}X^{\flat}\  \quad \mathrm{and} \quad \ \di\! \alpha_X  \,  = \,  - \frac{E}{m}X \, \lrcorner  \,\omega_J^2.
\end{split}
\end{equation}

\vspace{-.15cm}

We denote with $\Omega \in \Omega^4M$ the Kraines form of $Q$ that is $\Omega=\sum_a \o_a^2$ in a local basis $\{\o_a\}$ for $Q$.
Such a  basis is considered under the convention  $\o_a=g(I_a \cdot, \cdot)$ where the $g$-orthogonal almost complex structures $I_a$ satisfy 
$I_1I_2=-I_2I_1=I_3$. The Kraines  form $\Omega$ is $\nabla^g$-parallel with stabiliser algebra $\sp(1) \oplus \sp(n)$. Also recall the well known fact from quaternion-K\"ahler geometry that $R^g$ acts as a multiple of the identity on $Q$, explicitly $R^g\o_a=\Lambda_Q \o_a$ where $\Lambda_Q=\frac{mE}{m+4}$. In addition we have 
$$ 
R^g=\Lambda_{\bbE}\id \ \,\mathrm{on} \, \ \bbE \ \quad  \mathrm{and} \ \quad R^g=0 \  \, \mathrm{on} \, \ \mathbb{F}
$$
where $\Lambda_{\bbE}=\frac{4E}{m+4}$ and $\mathbb{F}$ indicates the orthogonal complement of $\bbE \oplus Q$ 
within $\Lambda^{1,1}_0M$. See \cite{SW}, Section 5.2, where the eigenvalues of the curvature operator $R^g:\Lambda^2M \to \Lambda^2M$ have been computed for every Wolf space $(M,g)$.

To set some more notation we indicate with $L_F$ the exterior multiplication with the $2$-form $F$. Its adjoint, computed with respect to  the metric $g$, reads $L_F^{\star}=\frac{1}{2}Fe_i \lrcorner e_i \lrcorner$.
\begin{lema} \label{Ot}
The  $4$-form  $\widetilde{\Omega} : =\omega_J^2+\frac{m-1}{3}\Omega$ is primitive with respect to the K\"ahler form  $\o_J$ and non-degenerate, 
i.e. the map $X \mapsto X \, \lrcorner \, \widetilde \Omega$ is injective,  if $m \geq 3$.
\end{lema}
\begin{proof}
Since $Q \subset \Lambda^{1,1}_0M$ we can repeatedly use Lemma \ref{Ls} to find 
$$
 L_{\o_1}^{\star}(\o_J \wedge \o_J)=-2\o_1, \quad  \ L_{\o_1}^{\star}(\o_a \wedge \o_a)=2\o_1 \ \; \mathrm{for} \ a=2,3.
$$
Since $[L_{\o_1}^{\star},L_{\o_1}]=(m-p) \, \id$ on $\Lambda^pM$ we also get $L_{\o_1}^{\star}(\o_1 \wedge \o_1)=2(m-1)\o_1$. All together, 
we obtain $L_{\o_1}^{\star}\Omega=2(m+1)\o_1$ and 
\begin{equation*}
L_{\o_1}^{\star}\widetilde{\Omega}=\frac{2(m^2-4)}{3}\o_1.
\end{equation*} 
In particular,   $\widetilde{\Omega}$ cannot vanish identically if $m\geq 3$. Note that in the same way one also sees that $\widetilde{\Omega}$ is primitive with respect to  $\o_J$, i.e. 
$L^{\star}_{\o_J}\widetilde{\Omega}=0$. This follows from $L^{\star}_{\omega_J}
( \omega_a \wedge \omega_a) = - 2 \omega_J$.
As  $\widetilde{\Omega}$ is $\nabla^g$-parallel and $g$ has irreducible holonomy either 
$\widetilde{\Omega}$ is non-degenerate or it vanishes. The latter case scenario cannot occur if $m \geq 3$ as seen above.
\end{proof}

The following proposition describes for the complex $2$-plane Grassmannian and its symmetric Einstein metric an explicit identification map between 
Killing vector fields and infinitesimal Einstein deformations. The existence of this isomorphism was already known by general arguments 
(see \cite[Theorem 1.1]{Ko3} or  \cite[Proposition 2.40]{GaGo}), but only with the explicit description given here it will be possible to compute the Koiso obstruction, as we will do below.

\begin{pro} \label{par1}
Consider the map $\varepsilon : \mathfrak{aut}(M,g) \to \Omega^{1,1}_0M$ defined by
$$
\varepsilon(X)=\; (\di\! X^{\flat})_0 \; - \;  \frac{(m-1)(m+4)}{3m} \,  (\di\!X^{\flat})_Q \ \quad  
$$ 
where the subscripts indicate orthogonal projection onto 
$\Lambda^{1,1}_0$ respectively $Q \subseteq \Lambda^{1,1}_0M$. Then  
\begin{itemize}
\item[(i)]
$
\varepsilon(X) \, :=\, (\di\! X^{\flat})_\bbE \; - \;  \frac{m^2 - 4}{3m} \,  (\di\!X^{\flat})_Q
$
\medskip
\item[(ii)] $\di \varepsilon(X) \, =  -\frac{E}{m} \, X \,  \lrcorner \, \widetilde{\Omega}$
\medskip
\item[(iii)]  $\varepsilon$ defines a linear isomorphism $\mathfrak{aut}(M,g) \cong \scrE^{+}(g) = \scrE(g)$, for $m\geq 3$.
\end{itemize} 
\end{pro}
\begin{proof}
(i) Recall that  $ (\di\! X^{\flat})_0 = (\di\! X^{\flat})_\bbE  +  (\di\!X^{\flat})_Q$. This directly implies the statement.

\noindent
(ii) For simplicity write $\beta_X=(\di\!X^{\flat})_Q$ and record that 
\begin{equation} \label{dB}
\begin{split}
&\di\!\beta_X = \frac{\Lambda_Q}{m}X \,  \lrcorner \, \Omega \ \quad  \mathrm{and} \  \quad \di^{\star}\!\beta_X=\frac{6\Lambda_Q}{m}X^{\flat}
\end{split}
\end{equation}
These relations follow immediately from the definitions and the following facts. The orthogonal projection onto $Q$ is given by $\alpha \in \Lambda^2M \mapsto \alpha_Q=\frac{1}{m}\sum_a g(\alpha,\o_a)\o_a \in Q$. In particular,
\begin{equation}\label{Qproj}
e_i \, \lrcorner \, (e^i \wedge X)_Q=\frac{3}{m}X^{\flat} \ \quad \mathrm{and} \ \quad  e^i \wedge (e^i \wedge X)_Q=-\frac{1}{2m}X  \, \lrcorner \, \Omega  .
\end{equation}
At the same time any Killing vector field $X$ satisfies $\nabla^g(\di\!X)=2R^g(X,\cdot)$. 
In addition we use that the bundle $Q$ is preserved by $\nabla^g$, i.e.  the covariant derivative of $\omega_a$ is a linear combination of
$\omega_b$ and $ \omega_c$ for $(a,b,c) $ a  permutation of $(1,2,3)$ with a skew-symmetric coefficient matrix. From this follows that the
computation of $\di \beta_X = e^k \wedge \nabla_{e_k}^g \beta_X$ and $\di^* \beta_X = - e_k \, \lrcorner \,   \nabla_{e_k}^g \beta_X$ is 
reduced to the two formulas in \eqref{Qproj}. See \cite[ Proposition 2.4]{Al} for details.

The claim on the exterior differential of $\varepsilon(X) = \alpha_X + t \beta_X$ follows now by combining equations \eqref{dA} and \eqref{dB}.

\noindent
(iii) For $t \in \mathbb{R}$ and $X \in  \mathfrak{g}$ we combine again \eqref{dA} and \eqref{dB} for checking that 
the quantity $\di^{\star}(\alpha_X+t\beta_X)=\frac{2}{m}(E(m-1)+3t\Lambda_Q)X^{\flat}$ vanishes precisely when 
$t=-\frac{E(m-1)}{3\Lambda_Q} = -  \frac{(m-1)(m+4)}{3m}$. 
Moreover, since the projection maps involved in the definition of $\varepsilon(X)$ are $\nabla^g$-parallel it follows that $\alpha_X$ and $\beta_X$ both belong to 
$\ker(\Delta^g-2E)$, as it is the case for $X$, thus $\varepsilon(X) \circ J \in \scrE^{+}(g)$. 

Now observe that $\varepsilon(X)=0$ forces $X \, \lrcorner \, \widetilde{\Omega}=0$ by (i). Since $\widetilde{\Omega}$ is non-degenerate it follows that $X=0$ thus 
$\varepsilon$ is injective and defines a surjective map onto some subspace of 
$\scrE^+(g) \subset \scrE(g)$.  As already mentioned above the spaces  $\aut (M, g) $ and $\scrE(g)$ are isomorphic by general arguments, i.e. they have in particular the same dimension. Therefore $\varepsilon$ defines  an explicit isomorphism as claimed.
\end{proof}

\begin{rema}
The same parametrisation of $\scrE(g)$ by Killing vector fields, as  given in  Proposition 4.2 in case of the complex $2$-plane
Grassmannian, is possible for all complex Grassmannians. Indeed, for all Grassmannians there are two parallel subbundles of
$\Lambda^{1,1}_0M$, as $\bbE$ and $Q$ in the case above. The curvature operator has two different, non-zero and constant eigenvalues
on these bundles and is zero on the complement. Moreover  the corresponding eigenbundles are algebra bundles. To find the suitable linear combination of the two projections of $\di\! X$ one
has to prove formulas similar to \eqref{Qproj}. A related identification seems to be contained in \cite{GaGo}, e.g. see Proposition 8.6 for a special
case. However  this reference uses a heavy harmonic analysis notation, whereas our approach is very direct, simple and explicit. 
\end{rema}

Yet another way of thinking about the map $\varepsilon$, which is going to be very useful in subsequent computations is obtained as follows. It also provides a computationally efficient way to deal with the projection on $\mathbb{E}$. As it is customary in quaternion-K\"ahler geometry consider the subbundle 
$$
\Lambda^2_{\sp}M:=\{F \in \Lambda^2M : [F,Q]=0\} \subseteq \Lambda^2M.
$$ 
Note that in the $E, H$ formalism of S. Salamon \cite{sala} we have $Q = (S^2H)_\RM$ and $\Lambda^2_{\sp}M = (S^2 E)_\RM$.

\medskip

\begin{lema} \label{int-eps}
The following hold
\begin{itemize}
\item[(i)] we have $\mathbb{E}=\Lambda^{1,1}_0M  \cap \Lambda^2_{\sp}M   $ 
\item[(ii)] assuming that $F \in \Lambda^{1,1}_0M$ we have that 
\begin{equation*}
\frac{1}{2}L_F^{\star}\widetilde{\Omega} \, = \, \frac{m-4}{3}F \,+ \, \frac{m(m-1)}{3}F_Q \, - \, \frac{4(m-1)}{3}F_{\bbE}
\end{equation*}
where $F_{\bbE}$ respectively $F_Q$ indicate the component of $F$ on $\bbE$ respectively $Q$.
\end{itemize}
\end{lema}
\begin{proof}
(i) is clear from  representation theoretic arguments but can also be proved directly as follows. Working at a given point in $M$, choose a basis $\{I_a,1 \leq a \leq 3\}$ in $Q$ and consider the $J$-invariant subspaces $V^{\pm}:=\ker(I_1J\pm 1)$. Then $TM=V^{+} \oplus V^{-}$ is a $g$-orthogonal
splitting and $V^{-}=I_2V^{+}$. Hence any $F \in \Lambda^{1,1}_0M \cap \Lambda^2_{\sp}M$ is uniquely determined by its restriction to $V^{+}$, which belongs to \, $\su(V^{+},g,J) \cong \su(n)$. In particular, $\Lambda^{1,1}_0M \cap \Lambda^2_{\sp}M$ and $\bbE \cong \su(n)$ have the same dimension. As $\bbE \subseteq \Lambda^{1,1}_0M \cap \Lambda^2_{\sp}M  $ since $[\bbE,Q]=0$, the claim is proved.

\noindent
(ii) Since $F \in \Lambda^{1,1}_0M$ direct computation yields $\frac{1}{2}L_F^{\star}\o_J^2=-F$. Similarly 
\begin{equation} \label{int-OO}
\frac{1}{2}L_F^{\star}\Omega=\sum_a g(F,\o_a)\o_a+\sum_a I_aFI_a=mF_Q+(-3F_{\bbE}+F_{\bbE^{\perp}})=F+mF_Q-4F_{\bbE.}
\end{equation}
where $\bbE^{\perp}$ is the orthogonal complement of $\bbE$ in $\Lambda^{1,1}_0M$. For the second equality we use the
following argument. First, let  $ \lambda_{I_a}M:=\{\alpha \in \Lambda^2M : \alpha(I_a \cdot, I_a\cdot)=-\alpha\}.$ Then we
have splitting 
$
\Lambda^2M=\Lambda^2_{\sp}M 
\oplus (\lambda_{I_1}M \cap \lambda_{I_2}M)
 \oplus (\lambda_{I_1}M \cap \lambda_{I_3}M) 
\oplus (\lambda_{I_2}M \cap 
\lambda_{I_3}M).
$
 Directly from here we have  that the operator $ \sum_a I_aFI_a$  equals $-3$ on the first summand and $1$ on the remaining  summands.
The claim follows from $\widetilde{\Omega}=\omega^2_J+\frac{m-1}{3}\O$.
\end{proof}
A special case of the above formula, which will be frequently used in what follows 
is 
\begin{equation} \label{s-case}
\frac{1}{2}L_F^{\star}\widetilde{\Omega }  \;  = - m F+\frac{(m-1)(m+4)}{3}F_Q \; \quad  \mathrm{for \ all} \quad F \in \bbE \oplus Q.
\end{equation}

Recall the well known fact that any $X \in \mathfrak{aut}(M,g)$ is quaternionic, i.e. $\L_X Q \subset Q$ or equivalently  
$\L_X\Omega=0$ (e.g. see \cite{Al}). The later condition follows for example 
from Proposition \ref{par1}, (ii) and the Cartan formula. Therefore the action of $\di\!X^{\flat}$ stabilises $\Omega$ and $\di\! X^{\flat}$ thus belongs to 
$\Omega_{\sp}^2M \oplus \Gamma(Q)$. This also follows from Kostant's theorem (see \cite[Theorem 4.5]{koba}) stating that $\nabla^g X$ is, at any point of $M$, an element of the holonomy algebra at that point. In particular  
\begin{equation} \label{type-X}
(\di\!X^{\flat})_0 \in \Gamma(\mathbb{E} \oplus Q) \quad  \mathrm{for \ all} \quad X \in \mathfrak{aut}(M,g).
\end{equation}

Comparing the definition of $\varepsilon (X)$ in Proposition \ref{par1} with \eqref{s-case} allows then to say that 
\begin{equation} \label{EE2}
\varepsilon(X) = - \frac{1}{2m} L_{(\di\!X^{\flat})_0}^{\star}\widetilde{\O}  = - \frac{1}{2m} L_{\di\!X^{\flat}}^{\star}\widetilde{\O}
\end{equation} 
The last equality is due to having $\widetilde{\Omega}$ primitive with respect to $\o_J$, i.e. to $L^{\star}_{\omega_J}\widetilde{\Omega}=0$
by Lemma \ref{Ot}. We thus obtain an explanation for the value needed for $t$ in the combination $(\di\!X^{\flat})_0+t(\di\!X^{\flat})_Q$ other than merely rendering that combination co-closed.

\subsection{Hamiltonian invariants within $\mathfrak{aut}(M,g)$} \label{hamI}
To compute the obstruction polynomial for elements in $\scrE(g)$ it is convenient to pause in order to describe 
a geometric way to compute various scalar invariants pertaining to elements $X \in \gg = \mathfrak{aut}(M,g)$. 
For the whole section we will be in the situation of complex $2$-plane Grassmannian, i.e. we can identify 
$\gg \cong \su(n+2)$.

\subsubsection{Invariant polynomials}\label{invariant}
In Theorem \ref{pol-fin}, one of the main results of this paper, we will identify  the obstruction polynomial for the $2$-plane Grassmannian with an invariant cubic
polynomial on the Lie algebra $\su(n+2)$. Here we recall a few facts on such polynomials. 

Any symmetric 
tensor $\mu \in S^k (\gg)$ defines via $\mu(X) = \mu(X, \ldots, X)$ a homogeneous polynomial of degree $k$ on $\gg$ and
invariance means  
$$\mu ([X,X_1], X_2, \ldots, X_k) + \ldots  + \mu (X_1, \ldots, [X, X_k]) = 0$$ 
for all vectors $X, X_1, \ldots, X_k \in \gg$. In what follows we will mainly deal with cubic polynomials on $\gg$, regarding which we record the following 
\begin{rema} \label{cub-su}
The space of invariant cubic forms on the Lie algebra $\su(n+2)$ is one dimensional by \cite[Prop. 2.1]{Gago2}. It is generated by the cubic form $P_0 \in S^3(\su(n+2))$ given by 
$$ 2P_0(A,B,C):=\tr(ACBJ_0+CABJ_0)$$
where $J_0$ is the standard complex structure on $\RM^{2(n+2)}$. 

This fact can be directly used on $M$ by identifying $\su(n+2)$ and $\mathfrak{g}$ via $A \mapsto X_A$, where $X_A$ is the Killing vector 
field  (fundamental vector field) corresponding to $A$ under the action of $\SU(n+2)$ on the complex $2$-plane Grassmannian. Thus, for any invariant cubic form $Q$ on $\gg$ we find a constant $c \in \bbR$ such that 
$Q(X_A,X_B,X_C) = c P_0(A,B,C)$ for all $A,B,C \in \su(n+2)$. For specific invariant cubics on $\gg$ such 
as $\mu_3$ and $\nu$ in this section and especially $P$ in section \ref{alg-t}, considerable work will be needed to determine such constants $c$ explicitly, in order to
decide whether or not they vanish.
\end{rema}

\begin{defi} \label{def-zloc}
The zero locus of  a cubic form $\mu \in S^3(\gg)$ is defined as the set of 
vectors $X \in \gg$ such that $\mu(X, X, Y) = 0$ for all vectors $Y \in \gg$. 
\end{defi}
Even with a non trivial zero locus a
cubic form $\mu$ may be non-degenerate, in the sense that  $\mu(X, \cdot , \cdot ) = 0$ implies $X=0$. Indeed, if 
$\mu$ is invariant the set $\{X \in \gg : \mu(X, \cdot, \cdot)=0\}$ is $\gg$-invariant, hence an ideal in $\gg$; as the latter Lie algebra is simple it follows that 
$\mu$ is either non-degenerate or it vanishes.

For the zero locus of $P_0$ we have following fundamental observation.
\begin{lema} \label{vanC}
\begin{itemize}
\item[(i)] The zero locus of the cubic polynomial $P_0$ is precisely the hyperquadric 
$\mathscr{C}(n)$  defined in \eqref{hq-i}
\item[(ii)] for $n$ odd we have that $\mathscr{C}(n) = \{  0 \}$.
\end{itemize}
\end{lema}
\begin{proof}
(i) From the definitions we see that the zero locus of $P_0$ is the set of matrices $A \in \su(n+2)$ such that $\tr (A^2 B J_0) =0$ for all matrices
$B \in \su(n+2)$. Hence the symmetric matrix $A^2 \in S^{2,+ } (\RM^{2(n+2)})$ is orthogonal to all trace free matrices $BJ_0 \in S^{2,+ }_0 (\RM^{2(n+2)})$
and thus has to be a multiple of the identity as claimed.\\
(ii) We have $A \in \mathscr{C}(n)$ if and only if $A^2 = \frac{1}{2(n+2)} \tr (A^2) \id$. If $A \neq 0$ we can normalise this equation to $A^2 = - \id$ which is
equivalent to $(AJ_0)^2 = \id$, as $A$ commutes with $J_0$. Then $AJ_0$ is symmetric and diagonalisable, having  eigenvalue $1$ 
with multiplicity $2p$ and eigenvalue $-1$ with multiplicity $2q$, with $p+q=n+2$. The  multiplicities  are even because of 
$[AJ_0, J_0] = 0$. Since $A \in \su(n+2)$ we have $\tr(AJ_0) = 0$ thus $p=q$. Then $n=2p$, i.e. $n$ has to be even and for $n$ odd it follows that $A=0$.
\end{proof}

The first invariant symmetric $k$-forms which we will use in this article  are $\mu_k :\gg \times \ldots \times \gg \to \bbR$ given by 
$$\mu_k(X_1, \ldots, X_k) = \int_M z_{X_1} \ldots z_{X_k} \vol
$$ for $k \in \mathbb{N}, k \geq 2$. 
We will be mainly interested in the cases  $k=2$ and $k=3$.
The invariance of $\mu_k$ follows by integration by parts from the relation  $z_{[X,Y]}=\L_Xz_Y$ and the
fact that $ \L_X $ for $X \in \gg $ is skew-symmetric with respect to the  $L^2$-inner product. 

Apart from $\mu_3$, a second cubic form will be of interest of interest in this paper; this time it takes into account the quaternion-K\"ahler structure of $(M,g)$ and it is defined as follows. We consider the $\gg$-invariant multi-linear map 
$$
 \nu : \gg \times \gg \times \gg \to \bbR, \; \; \; \nu(X_1,X_2,X_3):=\int_M g((\di\!X_1^{\flat})_Q, (\di\!X_2^{\flat})_Q )z_{X_3} \vol.
$$ 
and examine its algebraic properties. 

\begin{lema} \label{nu}
The map $\nu$ is symmetric in each pair of indices and defines a $\gg$-invariant cubic form in $S^3(\gg)$.
\end{lema}
\begin{proof}
We have 
\begin{equation*}
\nu(X_1,X_2,X_3)=\langle z_{X_3}(\di\!X_1^{\flat})_Q, (\di\!X_2^{\flat})_Q\rangle_{L^2}=\langle z_{X_3}(\di\!X_1^{\flat})_Q, \di\!X_2^{\flat}\rangle_{L^2}=\la \di^{\star}(z_{X_3}(\di\!X_1^{\flat})_Q),X_2^{\flat}\ra_{L^2}.
\end{equation*}
At the same time $\di^{\star}(z_{X_3}(\di\!X_1^{\flat})_Q)=z_{X_3}\di^{\star}(\di\!X_1^{\flat})_Q-
\grad z_{X_3}
\lrcorner (\di\!X_1^{\flat})_Q$. To bring this expression to final form recall that  $\di^{\star}(\di\!X_1^{\flat})_Q=\frac{6\Lambda_Q}{m}X_1^{\flat}$ by \eqref{dB}; we also have 
$$g(\grad z_{X_3}
\lrcorner (\di\!X_1^{\flat})_Q,X_2^{\flat})=g((\di\!X_1^{\flat})_Q,\di\!z_{X_3} \wedge X_2^{\flat})=
-g((\di\!X_1^{\flat})_Q, JX_3^{\flat} \wedge X_2^{\flat})
$$
since, in our conventions, $J\di\!z_{X_3}=X_3^{\flat}$. Collecting these facts leads to 
$$\nu(X_1,X_2,X_3)=\frac{6\Lambda_Q}{m}\langle X_1^{\flat},z_{X_3}X_2^{\flat} \rangle_{L^2}+
\langle (\di\!X_1^{\flat})_Q, JX_3^{\flat} \wedge X_2^{\flat}\rangle_{L^2}.
$$  
Skew-symmetrising in the last two slots thus yields 
\begin{equation*}
\begin{split}
\nu(X_1,X_2,X_3)-\nu(X_1,X_3,X_2)=&\frac{6\Lambda_Q}{m}\langle X_1^{\flat},z_{X_3}X_2^{\flat}-
z_{X_2}X_3^{\flat} \rangle_{L^2}\\
&+\langle (\di\!X_1^{\flat})_Q, 
JX_3^{\flat} \wedge X_2^{\flat}-JX_2^{\flat} \wedge X_3^{\flat}\rangle_{L^2}\\
=&\frac{6\Lambda_Q}{m}\langle X_1^{\flat},z_{X_3}X_2^{\flat}-
z_{X_2}X_3^{\flat} \rangle_{L^2}.
\end{split}
\end{equation*}
To obtain the vanishing of the second summand in the r.h.s. of the first equality above we have taken into account that 
the form $JX_3^{\flat} \wedge X_2^{\flat}-JX_2^{\flat} \wedge X_3^{\flat}$ is $J$-anti-invariant and thus satisfies  
$g((\di\!X_1^{\flat})_Q, JX_3^{\flat} \wedge X_2^{\flat}-JX_2^{\flat} \wedge X_3^{\flat})=0$ since 
$(\di\!X_1^{\flat})_Q$ belongs to $\Omega^{1,1}M$. Further on, and again due to $J\di\!z_{X_k}=X_k^{\flat}$ for $k=2,3$, we have 
$$g(X_1^{\flat},z_{X_3}X_2^{\flat}-
z_{X_2}X_3^{\flat})=g(\di\!z_{X_1},z_{X_3}\di\!z_{X_2}-z_{X_2}\di\!z_{X_3}).$$
From 
\begin{equation*}
\di^{\star}(z_{X_3}\di\!z_{X_2})=z_{X_3}\di^{\star}\di\!z_{X_2}-g(\di\!z_{X_2},\di\!z_{X_3})=2Ez_{X_2}z_{X_3}-
g(\di\!z_{X_2},\di\!z_{X_3})
\end{equation*}
we infer after skew-symmetrisation that $\di^{\star}(z_{X_3}\di\!z_{X_2}-z_{X_2}\di\!z_{X_3})=0$. Thus gathering these facts and by integration by parts it follows that $\nu(X_1,X_2,X_3)-\nu(X_1,X_3,X_2)=0$. Since $\nu$ is symmetric 
in the first two entries it follows that $\nu$ belongs to $S^3(\gg)$. 

The $\gg$-invariance of $\nu$ follows by a direct argument using integration by parts and that
\begin{itemize}
\item[$\bullet$]
any Killing vector field $X\in\gg$ is quaternionic, i.e.  the Lie derivative $\L_X$ commutes with the projection from $\Lambda^2M$ onto $Q$
\item[$\bullet$] we have $z_{[X,Y]}=\L_Xz_Y$ for all $X,Y \in \gg$.
\end{itemize} 
\end{proof}

\begin{rema} \label{id-mgen}
A similar argument as the one in the proof of Lemma \ref{nu} shows that on a compact K\"ahler-Einstein manifold $(M,g,J)$ we have 
$\int_M g(\di\!X_1^{\flat},\di\!X_2^{\flat})z_{X_3}\vol=0$ for all $X_1,X_2,X_3$ in $\mathfrak{aut}(M,g)$. This fact can provide a consistency check for some of the integral relations in Lemma \ref{int-eqp} below but will not be used directly in this paper.
\end{rema}
In Section \ref{z-locus} we will compare the cubic forms $\mu_3$ and $\nu$ and also determine their zero 
locus geometrically. In that section we will also provide a purely geometric construction, at manifold level, 
of the unique invariant cubic form in $S^3(\gg)$.
\subsubsection{Hermitian Killing forms}\label{HKF}
The fundamental observation making the computation of the obstruction polynomial possible will be that $(g,J)$ admits many explicit Hermitian Killing $2$-forms in the sense of \cite{Na1},
i.e. forms $F \in \Omega^{1,1}M$ such that $\nabla_U^gF = U \lrcorner \Psi + J(U \lrcorner \Psi)$ for some 
$3$-form $\Psi$. Explicitly 
we have
\begin{pro} \label{HK1}
Whenever $X \in \gg$ the form $\Phi_X := (\di\!X^{\flat})_{\bbE}-(\di\!X^{\flat})_Q - \frac{2E(m-4)}{m(m+4)}z_X\o_J$ in $\Omega^{1,1}M$ satisfies 
the Hermitian Killing form equation, i.e.
\begin{equation} \label{HK}
\nabla^g_U\Phi_X \; = \; U \, \lrcorner \,  \Psi_X \, + \, J(U  \, \lrcorner \, \Psi_X) \quad \mathrm{for \ all \ } \ U \in TM ,
\end{equation}
where $\Psi_X=-\frac{E}{2(m+4)}X \lrcorner (\O+\o_J^2)$.
\end{pro}
\begin{proof}
For any Killing vector field $X$ project the curvature equation $\nabla^g_U (\di\!X^{\flat}) = 2 R^g(X \wedge U)$
 onto $\bbE$ respectively $Q$. Recall that the subbundles $\bbE$ and $Q$ are preserved by $\nabla^g$.
 This yields
\begin{equation} \label{proj-K1}
\nabla_U^g(\di\!X^{\flat})_{\bbE} \, = \, 2\Lambda_{\bbE}(X^{\flat} \wedge U^{\flat})_{\bbE} \quad \mathrm{respectively} \quad 
\nabla_U^g(\di\!X^{\flat})_{Q} \, = \, 2\Lambda_{Q}(X^{\flat} \wedge U^{\flat})_{Q} ,
\end{equation}
where $\Lambda_{\bbE}$ and $\Lambda_{Q}$ are the eigenvalues of the curvature action on $\bbE$ respectively $Q$.
 Therefore 
$$
\nabla^g_U \Phi_X \; =\; \frac{2E}{m+4}\big (4(X^{\flat} \wedge U^{\flat})_{\bbE} \, - \, m(X \wedge U)_{Q} \big ) 
\;-\;  \frac{2E(m-4)}{m(m+4)}\omega_J(X, U) \omega_J .
$$ 
The rest of the proof is purely algebraic and unfolds as follows. Letting $\widehat{\O}:=\O+\o_J^2$ and 
$F:=X^{\flat} \wedge U^{\flat}+(JX)^{\flat} \wedge (JU)^{\flat}-\frac{2}{m}\o_J(X,U)\o_J \in \Omega^{1,1}_0M$ we use \eqref{int-OO}  and  $\frac{1}{2}L_F^{\star}\o_J^2=-F$ to get
$$
\frac{1}{2}L_F^{\star}\widehat{\O} \;= \; mF_Q \, - \,4F_{\bbE} \; =\;      2m (X^{\flat} \wedge U^{\flat})_Q   \, - \,8  (X^{\flat} \wedge U^{\flat})_{\bbE}.
$$
On the other hand $\widetilde{\O}$ is primitive so $L^{\star}_{\o_J}\widehat{\O}=2(m-4)\o_J$. Thus by also using the definition of the contraction 
with a $2$-form we  get 
$$ 
L_F^{\star}\widehat{\O} \;=\; U \lrcorner X \lrcorner \widehat{\O} \, + \, JU \lrcorner JX \lrcorner \widehat{\O} \, - \, \frac{4(m-4)}{m}\omega_J(X,U)\o_J.
$$
Note that we have that $JU \lrcorner JX \lrcorner \widehat{\O}=J(U \lrcorner X \lrcorner \widehat{\O})$, which holds since $\widehat{\O}$ 
is of bidegree $(2,2)$ with respect to $J$. Then combining these facts  $\nabla^g \Phi_X$ can be written as
$$
\nabla^g_U \Phi_X \; =\; - \frac{E}{2(m+4)} L_F^{\star}\widehat{\O} \;-\;  \frac{2E(m-4)}{m(m+4)}\omega_J(X, U) \omega_J 
\;=\; 
U \lrcorner \Psi_X \, + \, J(U \lrcorner \Psi_X) .
$$
\end{proof}

Note that other examples of Hermitian Killing forms on K\"ahler manifolds, however of different complex bidegree, have been obtained in 
\cite{NS-I}. Also note that 
$$
 \Phi_X\,= \, (\di\!X^{\flat})_{\bbE}  -  (\di\!X^{\flat})_Q \, = \, \varepsilon(X) \quad  \mathrm{when} \ m=4.
$$
Just as it is the case with the closely related notion of Hamiltonian $2$-forms, see \cite{ACG} for definitions and main properties, most 
relevant invariants of $\Phi_X$ can be derived directly from the Hermitian 
Killing equation in Proposition \ref{HK1}.
\begin{rema} \label{form-t}
For the detailed definitions of Killing forms, Hamiltonian 2-forms and Hermitian Killing forms and  
for explanations on how these notions are related we refer the interested reader to \cite{ACG,Na1,NS-I}.
\end{rema}

\subsubsection{Moment maps and Killing potentials}
The next important observation is that the forms $(\di\!X^{\flat})_{Q}$ and 
$(\di\!X^{\flat})_{\bbE}$, hence $\varepsilon(X)$ and $\Phi_X$, are Hamiltonian with respect to $X$.   The resulting 
moment maps turn out to be explicit as explained in the following 
\begin{lema} \label{Ham-new}
For any Killing vector field $X \in \gg$ the following hold.
\begin{itemize}
\item[(i)] We have \, $X \, \lrcorner \, (\di\!X^{\flat})_{Q} \, = \, \di\! \bfq_X$ \,  and \,   $X  \,\lrcorner  \,(\di\!X^{\flat})_{\bbE} \, = \, \di\!\bfe_X$  \, where 
\vspace{.1cm}
$$ 
\vert (\di\!X^{\flat})_Q \vert^2 \,=\, 4\Lambda_Q\mathbf{q}_X \quad \mathrm{and} \quad \vert (\di\!X^{\flat})_{\bbE} \vert^2 \, = \,4\Lambda_{\bbE}\mathbf{e}_X
$$

\vspace{.1cm}

\item[(ii)] 
The moment maps $\bfq_X$ and $\bfe_X$ satisfy 
\begin{eqnarray}
& \Delta^g\bfq_X&= \;\;  4\Lambda_Q\bfq_X  \, - \, \frac{6\Lambda_Q}{m}g(X,X) \label{Lap-23}\\[1ex]
& \Delta^g\bfe_X&=  \;\; 4\Lambda_{\bbE}\bfe_X  \,  - \,  \frac{\Lambda_{\bbE}(m^2-4)}{2m}g(X,X) \label{Lap-24}
\end{eqnarray} 

\vspace{.1cm}

\item[(iii)]
\quad $\int_M \bfq_X\vol \,  = \,  \frac{3E}{m}\mu_2(X,X)$  \quad and \quad  $\int_M \bfe_X\vol \, = \, \frac{(m^2-4)E}{4m}\mu_2(X,X)$

\vspace{.4cm}

\item[(iv)] 
\quad $g(X,X) \, + \, \frac{E}{m}z^2_X  \, + \,   \bfq_X \, + \,  \bfe_X  \;  = \;  \frac{(m^2+8m+12)E}{4m\vol (M)} \mu_2(X,X)$.
\end{itemize}
\end{lema}
\begin{proof}
(i) follows by taking the form inner product with 
$(\di\!X^{\flat})_Q$ respectively $(\di\!X^{\flat})_{\bbE}$ in the projected curvature equation  \eqref{proj-K1}, e.g. for $(\di\!X^{\flat})_Q$ we get
$
\frac12 \di |(\di\!X^{\flat})_Q |^2 = 2 \Lambda_Q \, X \, \lrcorner  \, (\di\!X^{\flat})_Q .
$

(ii) Co-differentiating in $X \lrcorner (\di\!X^{\flat})_{Q}=\di\!\bfq_X$  leads to the general formula 
\begin{equation*}
\Delta^g\bfq_X \, = \, - e_i \lrcorner \nabla^g_{e_i}X \lrcorner (\di\!X^{\flat})_Q \, - \, e_i \lrcorner X \lrcorner \nabla_{e_i}^g(\di\!X^{\flat})_Q
\, = \, \vert (\di\!X^{\flat})_Q\vert^2  -  X \lrcorner \di^{\star}(\di\!X^{\flat})_Q.
\end{equation*}
The claim in \eqref{Lap-23} follows from (i) and \eqref{dB}. To prove \eqref{Lap-24}, start from the general formula $\Delta^g\bfe_X=\vert (\di\!X^{\flat})_{\bbE}\vert^2-X \lrcorner \di^{\star}(\di\!X^{\flat})_{\bbE}$, which is proved as done above for $\bfq_X$ . Since we know that $(\di\!X^{\flat})_{\bbE}=
(\di\!X^{\flat})_0-(\di\!X^{\flat})_Q$ using \eqref{dA} and \eqref{dB} gives
$\di^{\star}(\di\!X^{\flat})_{\bbE}=\frac{2E(m^2-4)}{m(m+4)}X^{\flat}$ and the claim in \eqref{Lap-24} follows once again from (i).

(iii) As $(\Delta^g-2E)z_X=0$ we have $\vert X \vert^2_{L^2}=2E\mu_2(X,X)$, as already mentioned in the beginning of this section. Then both claims follows by integrating in 
the relations in (ii).

(iv) Since $X \lrcorner (\di\!X^{\flat})_0=-\di(g(X,X)+\frac{E}{m}z_X^2)$, as follows from $X \lrcorner \, \di\!X^{\flat} = - \di\!|X|^2$, 
the function $g(X,X)   +   \frac{E}{m}z_X^2+\bfq_X+\bfe_X$ must be constant over $M$ by (i). Recall again that
$(\di\!X^{\flat})_0 = (\di\!X^{\flat})_{Q}  + (\di\!X^{\flat})_{\bbE}$. The constant is then determined from the integrals of $\bfq_X$ and $\bfe_X$ computed in (iii).
\end{proof}

Finally we use the Hermitian Killing form $\Phi_X$ to establish the following proposition which plays an important role in investigating the structure of the invariant polynomials needed in the following section. 
\begin{pro}\label{Kpot2}
For any Killing vector field $X \in \gg$ the associated Hermitian Killing form $\Phi_X$, as  defined in Proposition \ref{HK1}, is Hamiltonian 
with respect to $X$. In addition we have $X \, \lrcorner \,  \Phi_X=\di\!\mathbf{p}_X$, where the function $ \mathbf{p}_X$ is defined as
$$
\mathbf{p}_X \, = \, \bfe_X  \, -  \,  \bfq_X  \, -  \, \frac{E(m-4)}{m(m+4)}\bigl (z_X^2+\frac{m^2+8m
+12}{4\vol(M)}\mu_2(X,X) \bigr ) .
$$ 
Moreover $ \mathbf{p}_X$ is a Killing potential on $M$, i.e. $J \grad  \mathbf{p}_X$ is a Killing vector field.
\end{pro}
\begin{proof}
The equation  $X\, \lrcorner\, \Phi_X=\di\!\mathbf{p}_X$ follows by definition and it implies
$JX \,\lrcorner \,  \Phi_X=-J\di\!\mathbf{p}_X$. Hence for $\mathbf{p}_X$  being a Killing potential
we have to show that  $\nabla^g(J\di\!\mathbf{p}_X)  = - \nabla^g (JX \, \lrcorner \,  \Phi_X)$ is skew symmetric.  This is done as follows.
From the Hermitian Killing equation \eqref{HK} and having $X \lrcorner \Psi_X=0$ we 
immediately obtain 
\begin{equation*}
\begin{split}
\nabla^g_U(JX \lrcorner \Phi_X) \; = \;  \nabla_U^g(JX) \,  \lrcorner \,  \Phi_X  \, + \,    JX \, \lrcorner \, \nabla^g_U  \Phi_X 
\; = \;   \nabla_U^g(JX)\,  \lrcorner \, \Phi_X \,  +  \,  JX \, \lrcorner \, U \,  \lrcorner\,  \Psi_X      .
\end{split}
\end{equation*}
The last summand is obviously skew symmetric, so we only have to discuss the first summand.
For the rest of the argument we will identify $2$-forms and skew-symmetric endomorphisms. 
Then $\di\!X^{\flat} \in \Lambda^{1,1}M $ commutes with $J$ and 
we have $\nabla_U^g(JX) \, \lrcorner \,\Phi_X = \frac{1}{2} (  \Phi_X \circ \di\!X^{\flat} \circ J)U$, where
we also used that the Killing vector field $X$ satisfies $\nabla^gX = \frac 12 \di\!X^{\flat}$  and that  $J$ is parallel.
Because $[\bbE,Q]=0$ in $\so(TM)$ we obtain
\begin{equation*}
\begin{split}
\Phi_X \circ \di\!X^{\flat} \circ J \;=\;  &-\frac{2E}{m}z_X\Phi_X \, + \, \frac{2E(m-4)}{m(m+4)}z_X((\di\!X^{\flat})_{\bbE}  \;  +  \;  (\di\!X^{\flat})_{Q})\\[1ex]
&
\qquad\qquad\qquad\qquad
+  \, (\di\!X^{\flat})_{\bbE} \circ (\di\!X^{\flat})_{\bbE} \circ J  \,  -  \,  (\di\!X^{\flat})_{Q} \circ (\di\!X^{\flat})_{Q} \circ J ,
\end{split}
\end{equation*}
where we decompose $\Phi_X$ according to its definition given in Proposition \ref{HK1} and we also write
$
\di\! X^{\flat} =  (\di\!X^{\flat})_{\bbE} +  (\di\!X^{\flat})_{Q} + \frac{2E}{m} z_X J
$.
It follows that $\Phi_X \circ \di\!X^{\flat} \circ J$ is skew-symmetric and thus 
$\nabla^g(J\di\!\mathbf{p}_X)$ is skew symmetric too. 
Hence $\mathbf{p}_X$ is a Killing potential as claimed. Note that part (iii) in Lemma \ref{Ham-new} is required in order to check that $\int_M \mathbf{p}_X\vol=0$.
\end{proof}

That $\mathbf{p}_X$ is a Killing potential can be also derived directly from \eqref{Lap-23} and \eqref{Lap-24} combined with the identity $\Delta^gz_X^2=4Ez_X^2-2g(X,X)$ and (iv) in Lemma \ref{Ham-new}. The proof above has 
the conceptual advantage of explaining the geometric origin of $\mathbf{p}_X$.

To finish this section we derive a few integral invariants which will be used in the next section to shed light 
on the invariant polynomial describing the obstruction to deformation of elements in $\scrE^{+}(g)$.
\begin{lema} \label{int-eqp}
For any  Killing vector fields $X,Y \in  \gg $ we have the following 
\begin{itemize}
\item[(i)] $(m-4)\int_M \bfe_Xz_Y\vol=-\frac{(m^2-4)E}{m}\int_Mz_X^2z_Y \vol$
\medskip
\item[(ii)] $(m-4)\int_M \bfq_Xz_Y \vol=3E\int_Mz_X^2z_Y \vol$
\medskip
\item[(iii)] $(m-4)\int_M \mathbf{p}_Xz_Y\vol=-\frac{Em(m+8)}{m+4}\int_Mz_X^2z_Y\vol$.
\end{itemize}
\end{lema}
\begin{proof}
First record the general fact
\begin{equation} \label{idXz}
\int_Mg(X,X)z_Y\vol  \,  =  \, E\int_Mz_X^2z_Y\vol ,
\end{equation}
which  actually holds for any K\"ahler-Einstein. The proof  follows  by integration by parts in the identity 
$g(X,X)z_Y=g(\di(z_Xz_Y),\di\!z_X)-\frac{1}{2}g(\di\!z_Y,\di\!z_X^2)$, since moment maps belong to $\ker(\Delta^g-2E)$.

\noindent
(i) and (ii) follow by first taking the $L^2$ scalar product with $z_Y$ in \eqref{Lap-23} and \eqref{Lap-24}, then using \eqref{idXz}.
Part (iii) follows from (i), (ii) and the definition of the Killing potential $\mathbf{p}_X$.
\end{proof}

\subsubsection{The zero locus of the invariant forms $\mu_3$ and $\nu$} \label{z-locus}
We will see that the zero locus of the cubic forms $\mu_3$ and $\nu$ plays a fundamental role in the  description of infinitesimal
Einstein deformations which are integrable to second order. To provide an explicit  geometric description of 
their zero locus we proceed as follows.

Define the multilinear map $\mathbf{P}:\gg \times \gg \times \gg \to \bbR$ according to 
\begin{equation*}
\mathbf{P}:=-\frac{1}{2\Lambda_Q}\nu-\frac{E(m+6)}{m+4}\mu_3
\end{equation*}
According to Lemma \ref{nu} and to the properties of $\mu_3$ outlined just above that Lemma we have that 
$\mathbf{P}$ defines a symmetric and invariant cubic form of $S^3(\gg)$. Below we relate this new cubic 
form to Killing potentials of the type $\mathbf{p}_X$, where is $X \in \gg$.
\begin{lema} \label{bfP}
We have 
\begin{equation*}
\mathbf{P}(X,X,Y)=\int_M\mathbf{p}_Xz_Y\vol
\end{equation*}
for all $X,Y$ in $\gg$.
\end{lema}
\begin{proof}
From the expression for the Killing potential $\mathbf{p}_X$ given in Proposition \ref{Kpot2} it follows that 
$\int_M\mathbf{p}_Xz_Y\vol=\int_M(\bfe_X-\mathbf{q}_X)z_Y\vol-\frac{E(m-4)}{m(m+4)}\int_Mz_X^2z_Y\vol$. Now take 
the $L^2$ scalar product with $z_Y$ in the identity in part (iv) of Lemma \ref{Ham-new} whilst keeping in mind that $\int_Mg(X,X)z_Y\vol=E\mu_3(X,X,Y)$, as granted by \eqref{idXz}. After a short calculation it follows that 
$\int_M(\bfe_X+\mathbf{q}_X)z_Y\vol=-\frac{E(m+1)}{m}\mu_3(X,X,Y)$. Using this we obtain 
$$ \int_M\mathbf{p}_Xz_Y\vol=-2\int_M\mathbf{q}_Xz_Y\vol-\frac{E(m+6)}{m+4}\mu_3(X,X,Y).
$$
The claim follows from having $\mathbf{q}_X=\frac{1}{4\Lambda_Q}\vert (\di\!X^{\flat})_Q \vert^2$, as guaranteed by 
Lemma \ref{Ham-new},(i).
\end{proof}
Now we are ready to show that $\mathbf{P}$ provides a geometric description, at manifold level, 
of the generator of the space of invariant elements in $S^3(\gg)$. This is intrinsic, in the sense that it does not use the identification $\gg \cong \su(n+2)$. Expressing $\mathbf{P}$ in terms of the local invariants 
$\mathbf{p}_X$ and $z_X$ is fundamental for the computation of the obstruction polynomial in section \ref{alg-t}.
\begin{pro} \label{nu-main}
The following hold
\begin{itemize}
\item[(i)] the zero locus of $\mathbf{P}$, that is $\{X \in \gg : \mathbf{P}(X,X,\cdot)=0 \}$, coincides with the quadric 
$$ 
\mathscr{C}:=\{X \in \gg : \mathbf{p}_X=0 \ \mathrm{at \ each \ point \ in} \ M\}  
$$
\item[(ii)] the cubic form $\mathbf{P}$ is not identically zero
\item[(iii)] the quadric $\mathscr{C} \subseteq \gg$ is isomorphic to the hyperquadric $\mathscr{C}(n)$ defined 
in \eqref{hq-i}.
\end{itemize}
\end{pro}
\begin{proof}
(i) By Lemma \ref{bfP} a Killing vector field $X$ is in the zero locus of $\mathbf{P}$ if and only
if $ \la \mathbf{p}_X,  z_Y \ra_{L^2} = 0$ for all $Y \in \gg$. Since $\mathbf{p}_X$ is a Killing potential this implies, as remarked at the
beginning of Section \ref{gra}, that $\mathbf{p}_X =0$, i.e. the zero locus of $\mathbf{P}$ equals  $ \mathscr{C}$.\\
(ii) We will show that $\mathbf{P}$ is not identically zero by showing that its zero locus $ \mathscr{C}$ is a proper subset of $\gg \cong \su(n+2)$. 
If $X \in \gg$ belongs to the isotropy algebra $\gg_p:=\{X \in \gg : X_p=0 \}$ at some point $p \in M$ then Lemma \ref{Ham-new}, (iv) ensures that 
$\frac{E}{m}z_X^2 +\bfq_X +\bfe_X  =\frac{E(m^2+8m+12)}{4m\vol(M)}\mu_2(X,X)
$ at $p$. If in addition 
$X$ belongs to $\mathscr{C}$, we eliminate $\mu_2(X,X)$  using the constraint $\mathbf{p}_X=0$ in order to arrive at $4\bfe_X-m\bfq_X=\frac{E(m-4)}{m}z_X^2$ at $p$. Since $\bfe_X$ and $\bfq_X$ are determined from 
$\di\!X^{\flat}$ as in Lemma \ref{Ham-new},(i)
 it follows that $\vert (\di\!X^{\flat})_{\bbE}\vert^2-\vert (\di\!X^{\flat})_Q\vert^2=
\frac{4E^2(m-4)}{m(m+4)}z_X^2$ at the point $p$.
Identifying the isotropy algebra $\gg_p$ and $\mathfrak{k}=\u(1) \oplus \sp(1) \oplus \su(n)$ we thus obtain an inclusion 
$$ 
\mathscr{C} \cap \mathfrak{k}  \subseteq \{(\lambda,A,B) \in \mathfrak{k}: \vert A \vert^2-
\vert B \vert^2=-\frac{4E^2(m-4)}{m(m+4)}\lambda^2 \} \neq \mathfrak{k}.
$$
It follows that $\mathscr{C} \neq \gg \cong \su(n+2)$, since otherwise we would have $ \mathscr{C} \cap \mathfrak{k} = \mathfrak{k}$.
Hence the cubic form $\mathbf{P}$ does not vanish identically.\\
(iii) The quadric  $\mathscr{C}$ occurs as the zero locus of a non identically zero, invariant cubic form 
on $\gg$, which is unique up to scale. Thus after identifying $\gg$ and $\su(n+2)$ as explained in Remark \ref{cub-su} the claim follows by using Lemma \ref{vanC},(i).
\end{proof}
As a direct consequence of the construction of the fundamental form $\mathbf{P}$ we obtain detailed information on 
the forms $\mu_3$ and $\nu$. As shown below, the complex dimension $m=4$ turns out to be rather special as the 
cubic polynomial $\mu_3$ vanishes in this dimension.  
\begin{pro} \label{mu3}
The following hold
\begin{itemize}
\item[(i)] If $m \neq 4$ then  the cubic polynomial $\mu_3$ is not identically zero. Its zero locus can be written as
$\{X \in \gg :  \la z_X^2,  z_Y \ra_{L^2} = 0 \ \mathrm{for \ all} \ Y \in \gg\}$ and it coincides with the quadric
$
\mathscr{C}$
\item[(ii)] If $m=4$  then  $\mu_3=0$. The cubic polynomial $\nu$ is not identically zero. Its zero locus is given by the
quadric $\mathscr{C}$.
\medskip
\end{itemize}
\end{pro}
\begin{proof}
(i) When $m \neq 4$ we have $\nu=\frac{12\Lambda_QE}{m-4}\mu_3$
as it follows by using in part (ii) of Lemma \ref{int-eqp} that $\vert (\di\!X^{\flat})_Q \vert^2=4\Lambda_Q \mathbf{q}_X$ and $\Lambda_Q=\frac{mE}{m+4}$. Having $\nu$ and $\mu_3$ thus proportional yields 
$\mathbf{P}=-\frac{Em(m+8)}{m^2-16}\mu_3$ and both claims on $\mu_3$ follow from Proposition \ref{nu-main}.

\noindent
(ii) When $m=4$ the vanishing of $\mu_3$ is granted by part (iii) in Lemma \ref{int-eqp}. Thus 
$\nu=-2\Lambda_Q\mathbf{P}$ and both claims follow again from parts (i) and (ii) in Proposition \ref{nu-main}.
\end{proof}
To finish this section a few remarks are in order.


\begin{rema}
For another proof of the non-vanishing of $\mu_3$ when $m\neq 4$, see \cite{H2}; in fact in that reference 
the non-vanishing of $\mu_3$ is proved for all Grassmannians $\mathrm{Gr}_p(\bbC^{p+q})$ with $p \neq q$. 
This result is proved in \cite{H2}
by a veritable {\it{tour de force}} in harmonic analysis using the Duistermaat-Heckman localisation formula. 
However, in our approach, using geometric analysis, the proof of these results is considerably simplified. In addition  \cite{H2} gives no geometrical
description of  the zero locus of $\mu_3$ as we do in Proposition \ref{mu3}. 

Finally, we note that the vanishing of $\mu_3$ when 
$m=4$ is also proved in \cite{H2} by a direct argument that works for all Grassmannians $\mathrm{Gr}_n(\bbC^{2n})$. 
It is based on the observation from \cite{GaGo} that $\mathrm{Gr}_n(\bbC^{2n})$ admits an anti-holomorphic involution.

\end{rema}
\subsection{The obstruction polynomial} \label{alg-t}
We wish to use Theorem \ref{E-geq0} in order to understand the obstruction theory of $g$. 
Recall that  it is enough to compute the cubic 
polynomial 
\begin{equation*}
\begin{split}
P(X) \; =  \;\; &3 \, \la \o_J \wedge \di\! \varepsilon(X), \varepsilon(X) \wedge \di\! \varepsilon(X)\ra_{L^2}  \,  -  \,8E \la \varepsilon(X)^2\circ J,\varepsilon(X)\ra_{L^2}\\
=
\; \;  & 3 \, \la \o_J \wedge \di\! \varepsilon(X), \varepsilon(X) \wedge \di\! \varepsilon(X)\ra_{L^2}  \, - \, 4E \la \varepsilon(X) \wedge \varepsilon(X), \varepsilon(X) \wedge \o_J \ra_{L^2}.
\end{split}
\end{equation*}
To see this, take $F=G = \varepsilon (X)$ in Theorem \ref{E-geq0}  
and use the parametrisation of $\scrE^+(g)$ given by the map $ X \mapsto \varepsilon(X) \circ J$. For the second equality we also use the formula of Lemma \ref{Ls}. Since for any $X \in \mathfrak{g}$ the exterior derivative 
$\di\!\varepsilon(X)$ is proportional to $X  \, \lrcorner \, \widetilde{\Omega}$ the problem of determining the obstruction polynomial is thus fully amenable to algebraic computation. Note that $P(X)$ coincides with the Koiso obstruction polynomial computed on $h = \varepsilon(X) \circ J$.


\subsubsection{Algebraic preparation of the computation}
Recall that for any Killing vector field $X \in \gg$ the Lie derivative $\L_X$ commutes with the projection onto $\Lambda^{1,1}_0M$, since $X$
is holomorphic, and with the projection onto $Q $,  since $X$ is quaternionic. Hence, $\varepsilon([X,Y])=\L_X \varepsilon(Y)$, i.e.
$\varepsilon$ is an iso\-mor\-phism of representations and we can identify $P$ with an $\ad$-invariant cubic polynomial on the Lie algebra $\su(n+2)$. 

As explained in Remark \ref{cub-su}, identifying $\su(n+2)$ and $\mathfrak{g}$ we find a constant $c \in \bbR$ such that
\begin{equation} \label{const}
P(X_A) =  c P_0(A,A,A) \quad  \mbox{for all} \ A \in \su(n+2)
\end{equation}
where $P_0$ is the basic  cubic forms on $\su(n+2)$ defined in Remark \ref{cub-su}.
Therefore the {\it{a priori}} answer to the deformation problem depends on the values for $c$. If $c=0$ every deformation is unobstructed, whilst if $c \neq 0$ the set of unobstructed deformations is isomorphic to the hyperquadric $\mathscr{C}(n)$.

Keeping this in mind, the remainder of this section is dedicated to the computation of the obstruction polynomial $P$ and of the constant $c$ above. Whenever $h \in S^2M$ we consider $\iota_h : \Lambda^{\star}M \to \Lambda^{\star}M$ defined  by 
$\iota_h(\alpha)=\sum he_k \wedge e_k \, \lrcorner \,  \alpha$ and note this is a symmetric operator. We also consider $\Pr :S^2M \to S^2M$ acting according to $\Pr(h)=\sum_a I_a \circ h \circ I_a$, whenever $\{I_a\}$ is a local basis in $Q$. Next we record that 
whenever $v \in TM$ and $h \in S^{2}_0M$ satisfies $hJ=Jh$ and $\tr(hJ)=0$ we have 
\begin{equation} \label{lin-rel}
\begin{split}
g(\iota_h(v \lrcorner \o_J^2),v \lrcorner \o_J^2)   &=  \;4(m-3)g(hv,v)\\[1ex]
g(\iota_h(v \lrcorner \o_J^2),v \lrcorner \Omega ) \; & =  \; -  12g(v,hv)  \,+ \,  8 g(v,\Pr(h)v)  \, + \, 8m 
g(hJ,(v \wedge Jv)_Q)   \\[1ex]
g(\iota_h(v \lrcorner \Omega),v \lrcorner \Omega )\;  &=  \; - 4 g(v,(3h+(m+4)\Pr(h))v ).
\end{split}
\end{equation}
The proof of the first two relations follows by straightforward algebraic computation using only 
the definitions. The proof of the third identity is contained in Lemma \ref{a-A1} of Appendix A.

\medskip

\begin{lema} \label{alg-1}
Assume that $v \in TM$ and $ F \in \mathbb{E} \oplus Q$. Then 
$$ 
g(\o_J \wedge ( v \,  \lrcorner \, \widetilde{\Omega}), F \wedge (v \, \lrcorner \, \widetilde{\Omega}))
\; = \; \, 4c_1g(F, v \wedge Jv) \, + \, 4c_2g(F_Q, v \wedge Jv)
$$
where $F_Q$ indicates the component of $F$ on $Q$ and the coefficients are
$$
-3c_1=(m-1)^2(m+3)-15m+9, \quad 9c_2=4(m-1)(m-4)(m+4).
$$
\end{lema}
\begin{proof}
Since $F$ is primitive by assumption and  $\widetilde{\Omega}$ is primitive by Lemma \ref{Ot} we find, with $h = FJ$, that
\begin{equation*}
\begin{split}
g(\o_J \wedge (v \, \lrcorner \, &\widetilde{\Omega}), F \wedge (v \, \lrcorner \, \widetilde{\Omega})) \; = \;
g(v \, \lrcorner \, \widetilde{\Omega},L^{\star}_{\o_J}(F \wedge  (v  \, \lrcorner \, \widetilde{\Omega})) \; = \;
g(v \; \lrcorner \; \widetilde{\Omega},\iota_h(v  \, \lrcorner \, \widetilde{\Omega}))\\
=\;\, &g(v \, \lrcorner  \, \o_J^2,\iota_h(v \,  \lrcorner \, \o_J^2)) \, + \, 
\frac{2(m-1)}{3}g(v \, \lrcorner  \,\O,\iota_h(v  \, \lrcorner  \,\o_J^2)) \, + \, \frac{(m-1)^2}{9}g(v \, \lrcorner \, \O,\iota_h(v \lrcorner \O)).
\end{split}
\end{equation*}
To obtain the last line we have expanded $\widetilde{\O}=\o_J^2+\frac{m-1}{3}\O$ and taken into account that $\iota_h$ is a symmetric operator.
Writing $F=F_\bbE +F_Q$ also observe that we have the identity $\Pr(h)=-3F_\bbE J+F_QJ=(-3F+4F_Q)J$. The claim thus follows by direct calculation from the relations 
in \eqref{lin-rel}.
\end{proof}
\subsubsection{Computation of the first summand in the obstruction polynomial}
One direct consequence of this algebraic lemma is that the only quantities needed to harness the first summand of the obstruction polynomial $P$ are 
$$\la (\di\!X^{\flat})_0,X^{\flat} \wedge (JX)^{\flat}\ra_{L^2} \ \mathrm{together \ with} \  
\la (\di\!X^{\flat})_Q,X^{\flat} \wedge (JX)^{\flat}\ra_{L^2}$$ where 
$X \in \gg$. In turn, these can be determined by using the moment maps $\bfq_X$ and $\bfe_X$ from the previous section. 
This observation is made precise below.
\begin{coro} \label{1st-half}
The following hold for all Killing vector fields $X \in \gg$
\begin{itemize}
\item[(i)] when $m \neq 4$ we have
$\la \o_J \wedge \di\! \varepsilon(X), \varepsilon(X) \wedge \di\! \varepsilon(X)\ra_{L^2}=c_3 \mu_3(X,X,X)$ where the constant  
$
c_3=-\frac{8E^4(m^2-4)}{m^3(m-4)}(2c_1+c_2)  
$
\medskip
\item[(ii)] when $m=4$ we have $\la \o_J \wedge \di\! \varepsilon(X), \varepsilon(X) \wedge \di\! \varepsilon(X)\ra_{L^2}=2E^2\nu(X,X,X).$
\end{itemize}
\end{coro}
\begin{proof}
Let $F:=\varepsilon(X)=(\di\!X^{\flat})_0+t(\di\!X^{\flat})_Q$; in particular we have 
$F_Q=-\frac{m^2-4}{3m}(\di\!X^{\flat})_Q$ and $F_{\bbE}=(\di\!X^{\flat})_{\bbE}$. 
Also recall that $\di\!F=-\frac{E}{m}X \, \lrcorner \,  \widetilde{\O}$ by Proposition \ref{par1},(ii). Then Lemma \ref{alg-1} gives
\begin{equation} \label{cons-1}
\begin{split}
g(\omega_J \wedge \di\!F,F \wedge \di\!F)  
\; = \;  \frac{4E^2}{m^2} \, \big (c_1g(F, X^{\flat} \wedge (JX)^{\flat}) \, + \, c_2g(F_Q, X^{\flat} \wedge 
(JX)^{\flat}) \big ) .
\end{split}
\end{equation}

\noindent
Assuming first that $m \neq 4$ we use  Lemma \ref{Ham-new}, (i) to get 
\begin{equation*}
\begin{split}
&\la (\di\!X^{\flat})_Q, X^{\flat} \wedge (JX)^{\flat} \ra_{L^2}=
\la X \lrcorner (\di\!X^{\flat})_Q,(JX)^{\flat}\ra_{L^2}=\langle \di\!\bfq_X,\di\!z_X\rangle_{L^2}=2E\int_M\bfq_Xz_X\vol\\
=&\frac{6E^2}{m-4}\mu_3(X,X,X)
\end{split}
\end{equation*}
with the last equality being due to Lemma \ref{int-eqp},(ii). Similarly, but using this time part (i) in Lemma \ref{int-eqp} at the end of the 
argument we get $\la (\di\!X^{\flat})_{\bbE}, X^{\flat} \wedge (JX)^{\flat}\ra_{L^2}=-\frac{2E^2(m^2-4)}{m(m-4)}\mu_3(X,X,X)$. Combining these two facts then yields 
\begin{equation} \label{prodds-1} 
\begin{split}
&\la F, X^{\flat} \wedge (JX)^{\flat} \ra_{L^2} = -\frac{4E^2(m^2-4)}{m(m-4)}\mu_3(X,X,X) \;\; \;\; \\
&\la F_Q, X^{\flat} \wedge (JX)^{\flat} \ra_{L^2}=-\frac{2E^2(m^2-4)}{m(m-4)}\mu_3(X,X,X)
\end{split}
\end{equation}
and the claim on the value of $c_3$ then follows. 

When $m=4$ we have $c_2=0$; in addition, $X \lrcorner (\di\!X^{\flat})_0=-\di(g(X,X)+\frac{E}{m}z_X^2)$ hence after successively using \eqref{idXz} and the vanishing of $\mu_3$ as granted by Proposition \ref{mu3} we get 
$$
\la (\di\!X^{\flat})_0, X^{\flat} \wedge (JX)^{\flat} \ra_{L^2}=-\la \di(g(X,X)+\frac{E}{m}z_X^2),\di\!z_X \ra_{L^2}=-2E
\int_M \big (g(X,X)z_X+\frac{E}{m}z_X^3\big )\vol=0.
$$
At the same time by using Lemma \ref{Ham-new}, (i) and  
$(JX)^{\flat} = \di\! z_X$ we get 
$$\la (\di\!X^{\flat})_Q, X^{\flat} \wedge (JX)^{\flat}\ra_{L^2}=\la \di\!\mathbf{q}_X, \di\!z_X\ra_{L^2}=2E\la \mathbf{q}_X,z_X\ra_{L^2}=\nu(X,X,X).$$ 
The last equality follows from $\vert(\di\!X^{\flat})_Q \vert^2=4\Lambda_Q\mathbf{q}_X$(see Lemma \ref{Ham-new}) and having $\Lambda_Q=\frac{E}{2}$ when $m=4$.

Since we also have $c_1=-4$ and $t=-2$, the claim finally follows from \eqref{cons-1}.
\end{proof}

\subsubsection{Computation of the second summand}
To finalise the computation of the obstruction polynomial $P$ we need to determine its second summand, that is 
$\la \varepsilon(X) \wedge \varepsilon(X), \varepsilon(X) \wedge \o_J \ra_{L^2}$.
This is done in two steps, starting with the following general lemma as a first step.

\begin{lema} \label{int-next2}
Let $(M^{2m},g,J)$ be a compact K\"ahler-Einstein with $\ric^g=Eg$ and $E>0$. For any
Killing vector field $X \in \mathfrak{aut}(M,g)$ and any infinitesimal Einstein deformation of the form 
$ F \circ J \in \scrE^{+}(g)$ we have 
$$
\la (\di\!X^{\flat})_0 \wedge (\di\!X^{\flat})_0,F \wedge \o_J\ra_{L^2}  \; =  \; 
 -\frac{2(m-4)E}{m} \la F, X^{\flat} \wedge (JX)^{\flat} \ra_{L^2} .
$$
\end{lema}
\begin{proof}
First record that the K\"ahler identities and having $F \in \Omega^{1,1}M$ co-closed ensure that 
$\di^{\star}L_{\o_J}F=J\di\!F$. Thus
$$
\la \di\!X^{\flat}  \wedge \di\!X^{\flat},F \wedge \o_J\ra_{L^2}=\la \di(X^{\flat} \wedge \di\!\!X^{\flat}), L_{\o_J}F\ra_{L^2}=\la X^{\flat} \wedge \di\!X^{\flat},J\di\!F
\ra_{L^2}=\la \di\!z_X \wedge \di\!X^{\flat},\di\!F \ra_{L^2}
$$
where we also used  $\di\!X^{\flat} \in \Omega^{1,1}M$ in the last equation.
In the same spirit, further integration by parts shows that 
$$
\la \di\!z_X \wedge \di\!X^{\flat},\di\!F \ra_{L^2} \; = \; 
\la \di (X^{\flat} \wedge \di\!z_X), \di\!F\ra_{L^2} \; = \; 2E \la X^{\flat} \wedge \di\!z_X, F\ra_{L^2}.$$
Because $F$ is primitive and $L^{\star}_{\o_J}L_{\o_J}F=(m-2)F$ we have 
$$ 
g((\di\!X^{\flat})_0 \wedge (\di\!X^{\flat})_0,F \wedge \o_J) \; = \; g(\di\!X^{\flat}  \wedge \di\!X^{\flat},F \wedge \o_J) \,  -  \,\frac{4(m-2)E}{m}
g(z_X\di\!X^{\flat},F).
$$
Again because $F$ is co-closed we have 
$$\la z_X\di\!X^{\flat},F\ra_{L^2}=\la \di(z_XX^{\flat})-
\di\!z_X \wedge X^{\flat}, F\ra_{L^2}=\la X^{\flat} \wedge \di\!z_X,F \ra_{L^2}$$ and the claim 
follows after substituting in the calculation from above.
\end{proof}
Returning to the complex Grassmannian set-up we come to the second step and  the main calculation which is based
on two algebraic observations. First we note
\begin{equation} \label{sq-Q}
L_{\o_J}^{\star} \bigl ((\di\!X^{\flat})_Q \wedge (\di\!X^{\flat})_Q \bigr )=-\frac{2}{m}\vert (\di\!X^{\flat})_Q\vert^2\o_J.
\end{equation}
This is established by computing $(\di\!X^{\flat})_Q$ in a local basis of $Q$ while keeping in mind that 
$L_{\o_J}^{\star}(\o_a \wedge \o_b)=0$ for $1 \leq a \neq b \leq 3$, whilst $L_{\o_J}^{\star}(\o_a \wedge \o_a)=-2\o_J $ for $1 \leq a \leq 3$. 
See also the proof of Lemma \ref{Ot} for algebraic arguments of this type. The next algebraic observation is

\begin{lema}\label{last}
Any Killing vector field $X \in \mathfrak{aut}(M,g)$ satisfies 
$$
g((\di\!X^{\flat})_0 \wedge (\di\!X^{\flat})_Q,\varepsilon(X) \wedge \o_J)=0.
$$ 
\end{lema}
\begin{proof}
First record the following algebraic facts. Due to \eqref{sq-Q} and having $Q$ a Lie algebra bundle we have 
$F \circ G \in Q \oplus \bbR \o_J$ for all $F,G$ in $Q$. Since $\bbE$ is the centraliser of $Q$ in $\su(TM)$, see (i) in Lemma \ref{int-eps}, we also have that $(F\circ G+G \circ F)J \in \bbE \oplus \bbR \o_J$ for all 
$F,G \in \bbE$. A direct orthogonality argument based on these facts leads to 
$$ F_E \circ F_Q \circ J \perp (\bbE \oplus Q) \ \mathrm{whenever} \ F_E \in \bbE \ \mathrm{and} \ F_Q \in Q.
$$
Hence if $F=F_{\bbE}+F_Q \in \bbE \oplus Q$ we have, according to \eqref{sq-Q} and by taking into account that $F_{\bbE},F_Q$ are primitive, 
$$L_{\o_J}^{\star}(F \wedge F_Q)=L_{\o_J}^{\star}(F_{\bbE} \wedge F_Q)-\frac{2}{m}g(F_Q,F_Q)J=2F_{\bbE} \circ F_Q \circ J-\frac{2}{m}g(F_Q,F_Q)J \perp \bbE \oplus Q.$$ Thus 
$g(F \wedge F_Q, \o_J \wedge G)=0$ whenever $F,G \in \bbE \oplus Q$ and the claim follows by taking $F=(\di\!X^{\flat})_0$
and $G=\varepsilon(X)$.
\end{proof}

Finally we are  able to compute the second summand of the obstruction polynomial.

\begin{coro} \label{2nd-half}
For any Killing vector field  $X \in \mathfrak{aut}(M,g)$ the following holds for $m \neq 4$
$$ 
\la \varepsilon(X) \wedge \varepsilon(X), \varepsilon(X) \wedge \o_J \ra_{L^2}
\; = \; 
\frac{8(m^2-4)E^3}{m^2}
\mu_3(X,X,X) .
$$

\vspace{-.7ex}

\noindent
In case $m=4$ the left hand side vanishes, i.e. $ \la \varepsilon(X) \wedge \varepsilon(X), \varepsilon(X) \wedge \o_J \ra_{L^2} = 0$.
\end{coro}
\begin{proof}
First we find 
$ \la \varepsilon(X) \wedge \varepsilon(X), \varepsilon(X) \wedge \o_J \ra_{L^2}
=
\la (\di\!X^{\flat})_0 \wedge (\di\!X^{\flat})_0, \varepsilon(X) \wedge \o_J \ra_{L^2}$
by combining   \eqref{sq-Q} and Lemma \ref{last}, and recalling that $\varepsilon(X)$ is primitive. 
Then Lemma \ref{int-next2} implies
$$ 
\la \varepsilon(X) \wedge \varepsilon(X), \varepsilon(X) \wedge \o_J \ra_{L^2}
\; = \; 
- \tfrac{2(m-4)E}{m} \la \varepsilon(X), X^{\flat} \wedge (JX)^{\flat} \ra_{L^2}.
$$
We see that the right hand side vanishes for $m=4$. In the case $m \neq 4$ the claim follows
from  the first equation in \eqref{prodds-1}.
\end{proof}

\subsubsection{The main result and proof of Theorem \ref{main5}}\label{rig}
We now have both summands of the obstruction polynomial realised as multiples of the invariant cubic polynomials $\mu_3$
for $m \neq 4$ respectively $\nu$ in the case $m=4$. Putting these formulas together and using the description of the zero locus
of the two cubic polynomials leads to the main result of this work: a complete solution of the deformation problem to second
order on the complex $2$-plane Grassmannian.

\begin{teo} \label{pol-fin}
Let $M^{2m}=\mathrm{Gr}_2(\bbC^{n+2})$ be equipped with its canonical Hermitian symmetric structure $(g,J)$. Then 
the following hold.
\begin{itemize}
\item[(i)] 
The obstruction polynomial can be written as
$$
P   =   \frac{16E^4(m^2-4)^2(m-1)}{3m^3(m-4)}\mu_3 \ \; \;\mathrm{if} \ m \neq 4
 \quad \mathrm{and} \quad
 P=6E^2\nu \ \; \; \mathrm{if} \ m=4
$$

\medskip

\item[(ii)] 
Any infinitesimal Einstein deformation $h \in \scrE(g) =  \scrE^{+}(g)$, written as $h = \varepsilon(X) J $ with  
a Killling vector field $X \in \gg$,
is integrable to second order if and only 
\begin{equation*}
\bfe_X-\bfq_X  \,  =   \,  \frac{E(m-4)}{m(m+4)}\bigl (z_X^2+\frac{m^2+8m
+12}{4\vol}\mu_2(X,X)
\bigr ) \ \; \mathrm{at\ each\ point\ in} \ M
\end{equation*}
or equivalently if $X \,  \lrcorner  \,\Phi_X=0$ over $M$.

\medskip

\item[(iii)] 
The set of infinitesimal Einstein deformations in $\scrE(g) \cong \su(n+2)$ which are integrable to second order is isomorphic to the hyperquadric $\mathscr{C}(n) \subseteq \su(n+2)$.
\end{itemize}
\end{teo}
\begin{proof}
(i) First assume that $m\neq 4$. Combining Corollary \ref{1st-half}, (i) and Corollary \ref{2nd-half} yields
$$ 
P(X)=-\frac{8E^4(m^2-4)}{m^3(m-4)} \big(3(2c_1+c_2)+4m(m-4) \big) \int_Mz^3_X \vol.
$$
The claim follows by expanding the polynomial expressions for $c_1$ and $c_2$, as given in Lemma \ref{alg-1}. 
When $m=4$ the expression for $P$ follows from Corollary \ref{1st-half}, (ii) together with Proposition \ref{2nd-half}.

\noindent
(ii) Clearly $P$
has no roots in $\mathbb{N}$. Thus $P$ is not identically zero and by Proposition \ref{mu3} it follows that  its zero locus is precisely the quadric 
$\mathscr{C}$. Hence an infinitesimal Einstein deformation $h= \varepsilon(X)J$ is integrable to second order if and only if the Killing potential
$ \mathbf{p}_X$ vanishes, which by Proposition \ref{Kpot2} is equivalent to the equation in (ii).

\noindent
(iii) follows from Proposition \ref{mu3}.
\end{proof}

The most striking application of our geometric description of infinitesimal Einstein deformations which are integrable to second order
is the following rigidity result.

\begin{teo} \label{rigid}
Assume that $n$ is odd. The symmetric Einstein metric on the complex $2$-plane Grassmannian $\mathrm{Gr}_2(\bbC^{n+2})$ is
isolated in the moduli space of Einstein metrics. All its infinitesimal Einstein deformations are obstructed to second order.
\end{teo}
\begin{proof}
When $n$ is odd we have $\mathscr{C}(n)=\{0\}$, a fact which has been proved in Lemma \ref{vanC}. Thus any element 
in $\scrE(g)$ is obstructed to second order hence there are no non-trivial curves of Einstein metrics through $g$. The claim now follows from the arguments in \cite[12.51]{Besse}.

In fact any Einstein metric suitably closed to $g$ in the sense of \cite[Definition 2.1]{Ko1} must be related to $g$ by a gauge transformation; this follows from \cite[Proposition 4.6]{Ko1}.
\end{proof}

\appendix

\section{Proof of the third relation in \eqref{lin-rel}} \label{formu-3}
Here we will work in a slightly more general set-up and consider $C:\End(TM) \to \End(TM)$ given by 
$Ch:=\sum_a I_ahI_a$. Direct algebraic computation based on the quaternion relations within the triple 
$\{I_1,I_2,I_3\}$ shows that 
\begin{equation} \label{A1}
C\{h,I_a\}=-\{Ch,I_a\}-2\{h,I_a\}
\end{equation}
for all $h \in \End(TM)$ and also that  
\begin{equation} \label{A2}
C^2=-2C+3.
\end{equation}
Below we prove a more general version of the third equation in \eqref{lin-rel}, in the sense that we do not assume that $h=FJ$ with $F$ in $\Lambda^{1,1}_0M$. 
\begin{lema} \label{a-A1}
Assuming that $(h,v) \in S^2_0M \oplus TM$ we have 
\begin{equation*}
g(\iota_{h}(v \lrcorner \Omega), v \lrcorner \Omega)=-4g(((m+4)Ch+3h)v,v).
\end{equation*}
\end{lema}
\begin{proof}
We have $v \lrcorner \Omega=2\sum_a (I_av)^{\flat} \wedge \omega_a$; since $\iota_{h}((I_av)^{\flat} \wedge \omega_a)=(hI_av)^{\flat} \wedge \omega_a+(I_av)^{\flat} \wedge \{h,I_a\}$ it follows that 
\begin{equation} \label{A23}
\begin{split}
g(\iota_{h}(v \lrcorner \Omega), v \lrcorner \Omega)=&2\sum_a g((hI_av)^{\flat} \wedge \omega_a+(I_av)^{\flat} \wedge \{h,I_a\},v \lrcorner \Omega)\\
=&2\sum_a g(v^{\flat} \wedge (hI_av)^{\flat},L^{\star}_{\omega_a}\Omega)+
2\sum_a g(v^{\flat} \wedge (I_av)^{\flat},L^{\star}_{\{h,I_a\}}\Omega).
\end{split}
\end{equation}
We now compute both summands independently. 
Using the formulas established during the proof of Lemma \ref{Ot} it follows that $L^{\star}_{\omega_a}\Omega=2(m+1)\omega_a$. Therefore 
\begin{equation} \label{A24}
\sum_a g(v^{\flat} \wedge (hI_av)^{\flat},L^{\star}_{\omega_a}\Omega)=2(m+1)\sum_a\omega_a(v,hI_av)
=-2(m+1)g((Ch)v,v).
\end{equation}
To compute the second summand, we first observe that using \eqref{A1} yields 
\begin{equation*}
\begin{split}
&g(v^{\flat} \wedge (I_av)^{\flat},
C\{h,I_a\})=-g(v^{\flat} \wedge (I_av)^{\flat},\{Ch,I_a\}+2\{h,I_a\})\\
=&-g(\{Ch,I_a\}v,I_av)-2g(\{h,I_a\}v,I_av)\\
=&-g((Ch)I_av+I_a(Ch)v,I_av)-2g(hI_av+I_ahv,I_av)\\
=&g((I_aChI_a)v,v)-g((Ch)v,v)+2g((I_ahI_a)v,v)-2g(hv,v)
\end{split}
\end{equation*}
after also expanding the anti-commutators and using that $I_a$ is a skew-symmetric isometry w.r.t. $g$. Summing over $a$ and taking \eqref{A2} into account leads to
\begin{equation*}
\sum_ag(v^{\flat} \wedge (I_av)^{\flat},C\{h,I_a\})=g((C^2h-Ch-6h)v,v)=-3g((Ch+h)v,v).
\end{equation*}
On the other hand, since $h$ is symmetric and trace free the form $\{h,I_a\}$ is orthogonal to $Q$, hence 
as in \eqref{int-OO} we get $L^{\star}_{\{h,I_a\}}\Omega=2C\{h,I_a\}$. Therefore 
\begin{equation*}
\sum_a g(v^{\flat} \wedge (I_av)^{\flat},L^{\star}_{\{h,I_a\}}\Omega)=2\sum_ag(v^{\flat} \wedge (I_av)^{\flat},
C\{h,I_a\})=-6g((Ch+h)v,v).
\end{equation*}
The claim follows by plugging \eqref{A24} and the above displayed equation into \eqref{A23}.
\end{proof}
Another way of presenting the identity in Lemma \ref{a-A1} is given below.
\begin{coro}We have
\begin{equation*} \label{rel-llf}
g(\iota_{h_1} \Omega, \iota_{h_2} \Omega)=-4(m+4)g(Ch_1,h_2)+12mg(h_1,h_2)
\end{equation*}
whenever $h_1,h_2$ belong to $S^2_0M$.
\end{coro}
\begin{proof}
First observe that 
$\iota_{h_1}(v \lrcorner \O)=v \lrcorner \iota_{h_1}\O-h_1v \lrcorner \O$ whenever $v \in TM$; by Lemma \ref{a-A1}  this leads to $ g(\iota_{h_1}\O, v^{\flat} \wedge (v \lrcorner \O))=g(h_1v \lrcorner \O,v \lrcorner \O)-4g(((m+4)Ch_1+3h_1)v,v).$

The tensor $(v,w) \mapsto g(v \lrcorner \O, w \lrcorner \O)$ is parallel hence it must be proportional to the metric $g$, since the latter has irreducible holonomy. To determine the proportionality factor recall that $L^{\star}_{\o_a}\O=2(m+1)\o_a$ from which we derive the equality $g(\O,\O)=6m(m+1)$. Thus, $g(v \lrcorner \O, w \lrcorner \O)=
12(m+1)g(v,w)$. These facts ensure that 
$$ g(\iota_{h_1}\O, v^{\flat} \wedge (v \lrcorner \O))=-4(m+4)g(Ch_1v,v)+12mg(h_1v,v).$$
In particular, since $4\O=\sum_i e^i \wedge e_i \lrcorner \O$ whenever $\{e_i\}$ is an orthonormal basis in $TM$ and $h_1,Ch_1$ are trace-free, we get $g(\iota_{h_1}\O,\O)=0$. The claim follows now by polarisation 
and using decomposable elements of the form $h_2=v^{\flat} \otimes w+w^{\flat} \otimes v-\frac{1}{m}
g(v,w)\id$ with $v,w \in TM$.
\end{proof}
\section{The Koiso obstruction polynomial}\label{koiso-plus}

In this section  we want to relate the integrand $P(h)$ of the Koiso obstruction of an element $h \in \scrE(g)$ to  
the divergence of the Fr\"olicher-Nijenhuis bracket
$ [h,h]^{\FN}$. This will also give a direct proof of Proposition \ref{div-bra}.
Recall that $P(h)$ was defined via
 \begin{equation*}
2P(h) \;:= \; 3  g(\nabla^2_{e_i,he_i}h,h) \, - \, 6 g ((\nabla^2_{e_i,e_j}h)he_i,he_j) \, + \, 2 E \, \tr(h^3),
\end{equation*} 
%
In order to simplify the notation we will use $\nabla$ instead of $\nabla^g$ within this appendix.
We start with differentiating in \eqref{bra-na} and obtain
\begin{equation}\label{bracket}
\begin{split}
\delta^g [h,h]^{\FN}_X \; = &\;  (\nabla^2_{e_i,he_i}h)X \,  - \,   (\nabla^2_{e_i,hX}h)e_i  \, -  \, (\nabla_{(\nabla_{e_i}h)X}h)e_i\\[1ex]
&
\qquad\qquad\qquad
 - \,  (\nabla_{e_i}h)^2 X   \, +  \,  (\nabla_{e_i}h) (\nabla_Xh)e_i \, +  \, h(\delta^g \di_{\nabla}h)_X.
\end{split} 
\end{equation}
We now compute separately the $L^2$ inner product of each summand above with $h$. 
The scalar product with  the first summand on the right hand side  already gives the first summand in  $P(h)$.
From the Ricci identity and $\delta^g h=0$ we find for the second summand that
\begin{equation*} \label{inter-3}
- g( (\nabla^2_{e_i,he_j}h)e_i,he_j ) \; = \; g(  (\ring{R}h)h,h) \, -  \,  E \, \tr(h^3).
\end{equation*}
The scalar product with the fourth summand can be expressed using the identity 
$$
 \nabla^{\star}\nabla h^2 \; = \; (\nabla^{\star}\nabla h) \circ h \, + \, h \circ (\nabla^{\star}\nabla h) \, - \, 2(\nabla_{e_i}h)^2
$$
and the fact that $\nabla^{\star}\nabla$ is self adjoint we obtain the  expression 
\begin{equation*} 
- \, \langle (\nabla_{e_i}h)^2,h\rangle_{L^2}  \;=\;  - \frac{1}{2}\langle \nabla^{\star}\nabla h,h^2\rangle_{L^2}.
\end{equation*}
The scalar product in the third summand can be expressed in the following way
\begin{equation*}
\begin{split}
g( (\nabla_{(\nabla_{e_i}h)e_j}h)e_i,he_j) \;=  \; &\; g( (\nabla_{e_i}h)e_j,e_k) \,
g( (\nabla_{e_k}h)e_i,he_j) \;= \; g( (\nabla_{e_k}h)e_i,h(\nabla_{e_i}h)e_k) \\[1ex]
=&-g( (\nabla_{e_k}h)e_i,(\nabla_{e_i}h)he_k) +\langle (\nabla_{e_k}h)e_i,(\nabla_{e_i}h^2)e_k\rangle.
\end{split}
\end{equation*}
The second summand in the last line can be modified by introducing a divergence term:
$$
 g( (\nabla_{e_k}h)e_i,(\nabla_{e_i}h^2)e_k )  \,  + \,  g(
(\nabla^2_{e_i,e_k}h)e_i,h^2e_k ) \; = \; -\di^{\star}(X \mapsto \langle (\nabla_{e_k}h)X,h^2e_k\rangle)
$$
after integrating by parts and using the Ricci identity on the second summand above we find 
\begin{equation*} \label{inter-2}
\begin{split}
\langle (\nabla_{(\nabla_{e_i}h)e_j}h)e_i,he_j\rangle_{L^2} \; = & \; - \langle (\nabla_{e_k}h)e_i,(\nabla_{e_i}h)he_k\rangle_{L^2}
\, + \, \langle \ring{R}h,h^2 \rangle_{L^2} \, - \, E\int_M\tr(h^3)\\
=  & \quad
\la (\nabla^2_{e_i,e_j}h)he_i,he_j\ra_{L^2}
\, + \, \langle \ring{R}h,h^2 \rangle_{L^2} \, - \, E\int_M\tr(h^3) ,
\end{split}
\end{equation*}
where the second summand of the Koiso obstruction $P(h)$ appears on the right hand side due to the equation
$g( (\nabla^2_{e_i,e_j}h)he_i,he_j) = - g( (\nabla_{e_j}h)e_i,(\nabla_{e_i}h)he_j) - \di^{\star}(X \mapsto g( (\nabla_{e_j}h)hX,he_j))$,
Note that this also relates the second summand in $P(h)$ and  the scalar product with the fifth summand in \eqref{bracket}.
Putting things together leads to 
\begin{equation*}
\langle \delta^g [h,h]^{\FN},h\rangle_{L^2}
\; = \;
\langle (\nabla^2_{e_i,he_i}h),h\rangle_{L^2} 
\, - \, 
2 \la  ((\nabla^2_{e_i,e_j}h)he_i,he_j \ra
\, + \, 
\langle \delta^g \di_{\nabla}h-\frac{1}{2}\nabla^{\star}\nabla h,h^2\rangle_{L^2}.
\end{equation*}
Hence after integration and with the Weitzenb\"ock formula  $\delta^g \di_{\nabla}=\nabla^{\star}\nabla-\ring{R}+E$ we find 
\begin{equation*}
\begin{split}
2\int_M P(h)\vol \; =  & \; 3 \, \langle \delta^g [h,h]^{\FN},h\rangle_{L^2} \, - \, 3 \, \langle \delta^g \di_{\nabla}h-\frac{1}{2}\nabla^{\star}\nabla h,h^2\rangle_{L^2}
\, + \, 2E\int_M \tr(h^3)\vol \\[.5ex]
= & \;
3 \, \langle \delta^g [h,h]^{\FN},h\rangle_{L^2} \, - \, \frac32 \,  \la  \Delta_E h, h^2 \ra_{L^2} \, - \, E\int_M \tr(h^3)\vol  .
\end{split}
\end{equation*}

\end{document}